\documentclass[a4paper,11pt]{article}
\usepackage{amssymb, amsmath, amsthm, graphicx}

\newtheorem{theorem}{Theorem}[section]

\newtheorem{lemma}{Lemma}[section]
\newtheorem{proposition}{Proposition}[section]
\newtheorem{corollary}{Corollary}[section]

  \makeatletter
  \@addtoreset{equation}{section}
  \makeatother
\usepackage[top=40truemm,bottom=40truemm,left=30truemm,right=30truemm]{geometry}

\title{\bfseries{\Large{Generation and motion of interfaces in one-dimensional stochastic Allen-Cahn equation}}}
\author{\Large{Kai Lee}\\
Graduate School of Mathematical Sciences, University of Tokyo
}

\date{}

\begin{document}

\maketitle

\begin{abstract}
In this paper we study a sharp interface limit for a stochastic reaction-diffusion equation which is parametrized by a sufficiently small parameter $\varepsilon >0$. We consider the case that the noise is a space-time white noise multiplied by $\varepsilon ^\gamma a(x)$ where the function $a(x)$ is a smooth function which has a compact support. At first, we show a generation of interfaces for a one-dimensional stochastic Allen-Cahn equation with general initial values. We prove that interfaces are generated in a time of order $O(\varepsilon |\log \varepsilon|)$. After the generation of interfaces, we connect it to the motion of interfaces which was investigated by Funaki \cite{f} for special initial values. Funaki \cite{f} proved that the interface moved in a proper time scale obeying a certain SDE if the interface formed at the initial time. We take the time scale of order $O(\varepsilon^{-2\gamma - \frac{1}{2}})$. This time scale is same as that of \cite{f} and interface moves in this time scale obeying some SDE with high probability.
\end{abstract}

\section{Introduction}
\label{sec1}
The Allen-Cahn equation is a reaction diffusion equation with a bistable reaction term $f$; see (\ref{reaction}) for detailed conditions on $f$. This equation describes physical phenomena such as dynamical phase transition, and, in one dimension, it has the form:
\begin{align}
\label{eq:pde}
\begin{cases}
\dot{u} ^\varepsilon (t,x) &= \Delta u^\varepsilon (t,x)+\displaystyle{\frac{1}{\varepsilon}} f(u^\varepsilon (t,x) ),\ \ \ t > 0,\  x\in \mathbb{R},\\
u^\varepsilon (0,x) &= u_0 ^\varepsilon (x),\ \ \ x\in \mathbb{R},
\end{cases}
\end{align}
where $\varepsilon >0$, $\dot{u} = \frac{\partial u}{\partial t}$ and $\Delta u=\frac{\partial ^2 u}{\partial x^2}$. We assume that the function $f$ has $\pm 1$ as stable points and satisfies $\int _{-1} ^1 f(u) du=0$. Then, it is expected that the solution $u^\varepsilon$ tends to $\pm 1$ as $\varepsilon \to 0$ in a very short time and an interface appears to separate two different phases $\pm1$. In recent studies of the deterministic case, the behaviors of the solution have been investigated. For example, Chen \cite{xc} studied the initial value problem (\ref{eq:pde}) in one dimension and classified the behaviors of solutions into four stages:
(i) Phase separation: In a very short time $u^\varepsilon$ tends to $\pm 1$. In other words, interfaces are generated in a time of order $O(\varepsilon |\log \varepsilon|)$.
(ii) Generation of metastable patterns: Until the time of order $O(1)$, $u^\varepsilon$ enters into a neighborhood of standing waves associated with $f$.
(iii) Super-slow motion of interfaces: An approximated ODE governs the very slow interface motion for a long time of order $O(e^\frac{C}{\varepsilon})$ with $C>0$.
(iv) Annihilation of interfaces: Under the super-slow motion, when two interfaces are close enough, the interface between them is annihilated and they restore the super-slow motion.

We are interested in the first generation time of interfaces and an appropriate time scale for the interface motion when a random external noise term is added. Carr and Pego \cite{cp} studied the one-dimensional deterministic case, and they proved the proper time scale for interface motion is of order $O(\exp(\frac{C}{\varepsilon}))$ as we mentioned above. Funaki \cite{f94} and \cite{f} studied the stochastic case with an additive noise:
\begin{align}
\label{eq:spde}
\begin{cases}
\dot{u} ^\varepsilon (t,x) &= \Delta u^\varepsilon (t,x)+\displaystyle{\frac{1}{\varepsilon}} f(u^\varepsilon (t,x) ) + \varepsilon ^\gamma a(x) \dot{W} _t(x),\ \ \ t> 0,\  x\in \mathbb{R},\\
u^\varepsilon (0,x) &= u^\varepsilon _0 (x),\ x\in \mathbb{R}, \ \ \ u^\varepsilon (t,\pm \infty) = \pm 1,\ \ \ t\geq 0,
\end{cases}
\end{align}
where $a\in C_0^\infty (\mathbb{R})$. Here $\dot{W} _t(x)$ is a space-time white noise on $\mathbb{R}$ which formally has a covariance structure
\begin{align}
\label{cov}
E[\dot{W}_t (x) \dot{W}_s (y)] = \delta(t-s) \delta(x-y)
\end{align}
and $\delta$ is the Dirac's delta (See also Bertini et al. \cite{bbb} and \cite{bbbp}). Funaki \cite{f} showed that the proper time scale is of order $O(\varepsilon^{-2\gamma - \frac{1}{2}})$. This behavior of the solution is corresponds to the phase (iii) in the deterministic case. The motion of interface for the stochastic case is much faster than that of the deterministic case only in this phase because of the strong effect of the noise. Funaki treated the case that an interface is already formed at the initial time. 

In this paper, we investigate more general initial values and, in particular, compute the first generation time of the interface. We further study whether we can connect it to the motion of interface in the case that the initial value is not an interface. 

\subsection{Setting of the model}
We consider the SPDE (\ref{eq:spde}) of Allen-Cahn type in one dimension. The reaction term $f\in C^2( \mathbb{R} )$ satisfies the following conditions:
\begin{align}
\label{reaction}
 \begin{cases}
  \text{(i)}f \text{ has only three zeros }\pm 1, 0, & \\
  \text{(ii)}f'( \pm 1) =:-p < 0,\  f'(0)=: \mu > 0, & \\
  \text{(iii)}f(u)\leq C(1+|u|^q)\text{ with some }C,q>0,& \\
  \text{(iv)}f'(u)\leq c\text{ with some }c>0, & \\
  \text{(v)}f\text{ is odd,} &\\
  \text{(vi)}f(u)\leq -p(u-1) \ (u\geq 1).
 \end{cases}
\end{align}
The conditions (i) and (ii) imply that the reaction term is bistable and has only $u=\pm 1$ as stable points. The existence of the global solution for the SPDE (\ref{eq:spde}) is assured by (iii) and (iv) (see p.222 and Section 2 of \cite{f}, and Section 2 of \cite{f94}). Moreover, we need the assumption (iv) in order to prove a comparison theorem by applying the maximum principle for the parabolic PDEs (see Section 2 of \cite{fr}). The condition (v) implies $\int _{-1}^{1}f(u)du=0$, from which we see that the corresponding traveling wave solution is actually a standing wave. We impose (vi) for a technical reason. We can take $f(u)=u-u^3$ as an example of $f$.

Next we explain the external noise term. At first, we fix a filtered probability space $(\Omega, \mathcal{F}, P, \{ \mathcal{F} _t\} _{t\geq 0})$ and consider stochastic processes defined on it. Let $\dot{W}_t (x)$ be the space-time white noise which formally has the covariance structure (\ref{cov}) and is an $\{ \mathcal{F} _t\} _{t\geq 0}$-adapted process. We can rewrite the equation (\ref{eq:spde}) in the mild form:
\begin{align}
u^\varepsilon(t)= S_t u_0 ^\varepsilon + \frac{1}{\varepsilon} \int _0^t S_{t-s}f(u^\varepsilon (s))ds +\ \int _0^t S_{t-s} \varepsilon ^\gamma a dW_s \nonumber
\end{align}
where $S_t$ is an integral operator defined by $S_t u(x):=\int _{\mathbb{R}} p(t,x,y)u(y)dy$ and $p(t,x,y) :=\frac{1}{\sqrt{4\pi t}}e^{-\frac{(x-y)^2}{4t}}$. We give a mathematical meaning to the last term as a stochastic integral with respect to an operator valued integrand. Another way to interpret (\ref{eq:spde}) is as a weak solution, namely $u^\varepsilon(t)$ satisfies
\begin{align}
\langle u^\varepsilon(t) - u^\varepsilon(0) , \varphi \rangle =\int _0 ^t \langle u^\varepsilon(s) , \Delta \varphi \rangle ds + \frac{1}{\varepsilon} \int _0 ^t \langle f(u^\varepsilon(s)) , \varphi \rangle ds + \varepsilon ^\gamma \int _0 ^t \langle \varphi , adW_s \rangle \nonumber
\end{align}
for all $\varphi \in C_0 ^\infty (\mathbb{R})$. Here $ \langle , \rangle$ means the inner product on $L^2(\mathbb{R})$. It is well-known that every mild solution is a weak solution and vice versa (see \cite{dpz}). Moreover we assume that $u_0 ^\varepsilon \in C^2(\mathbb{R})$ and there exist constants $C_0 >1$, $C$, $C'>0$ and $\kappa >1$ such that
\begin{align}
\label{eq:ini}
 \begin{cases}
   \text{(i)} \| u_0 ^\varepsilon \| _\infty + \| u_0 ^{\varepsilon \prime}  \| _\infty + \| u_0 ^{\varepsilon \prime \prime} \| _\infty \leq C_0, & \\
   \text{(ii)} \text{There exists a unique } \xi _0 \in [-1,1] \text{ independent of } \varepsilon >0 \text{ such that } u_0 ^\varepsilon (\xi _0)= 0, & \\
   \text{(iii)} | u_0 ^\varepsilon (x)| \geq C \varepsilon ^{\frac{1}{2}}\ (|x-\xi _0| \geq C' \varepsilon ^{\frac{1}{2}}), & \\
   \text{(iv)} | u_0 ^\varepsilon (x) - 1 | + | u_0 ^{\varepsilon \prime} (x) | + | u_0 ^{\varepsilon \prime \prime} (x) | \leq \varepsilon ^{\kappa }C_\mu \exp(-\frac{\sqrt{\mu} x}{2}) \  (x\geq 1), & \\
   \text{(v)} | u_0 ^\varepsilon (x) + 1 | + | u_0 ^{\varepsilon \prime} (x) | + | u_0 ^{\varepsilon \prime \prime} (x) | \leq \varepsilon ^{\kappa } C_\mu \exp(\frac{\sqrt{\mu} x}{2}) \  (x\leq  -1), & \\
 \end{cases}
\end{align}
where $C_\mu:=\frac{\mu}{4} \wedge 1$ and $\| \cdot \| _\infty$ is the supremum norm on $C(\mathbb{R})$. Here the constant $\mu$ is defined in (\ref{reaction}). Conditions on the constant $\kappa >1$ are stated in Theorem \ref{thm23} and in Section \ref{sec3}. We use the assumption (i) throughout of this paper. Because we consider the case that only one interface is formed, we assume the condition (ii). We use the conditions (iii), (iv) and (v) in order to prove the generation of interface for the deterministic case in Section \ref{sec2} as a preparation.

In this paper, we assume that the support of $a$ is included in $[-1,1]$ without loss of generality. And for each $n \in \mathbb{N}$, Sobolev space $H^n(\mathbb{R})$ is defined by $H^n(\mathbb{R}) := \{ f\in L^2(\mathbb{R}) | \|f \|_{H^n(\mathbb{R})} < \infty \}$ equipped with a norm $\|f \|_{H^n(\mathbb{R})} := \sum _{k=0} ^n \|\nabla ^k f\|_{L^2(\mathbb{R})}$ where $\nabla ^k f (x) := \frac{d^k f}{dx^k}(x)$.

\subsection{Main result}
As we mentioned, in this paper, we discuss the generation of interfaces and give estimates on the first generation time of interfaces. After that, we connect this to the motion of interface in one-dimensional case which was introduced in \cite{f}. Before we state the main result, we define a function $m$ which satisfies the following ODE and is called a standing wave:
\begin{align}
\label{ode11}
\begin{cases}
\Delta m + f(m) =0,\ m(0)=0,\ m(\pm \infty) =\pm 1, \\
m\text{ is monotone increasing}.
\end{cases}
\end{align}
We explain about this function below. Now we formulate our main result.
\begin{theorem}
\label{thm23}
Assume that $u_0^\varepsilon$ satisfies (\ref{eq:ini}),  $\bar{u}^\varepsilon (t,x) := u^\varepsilon (\varepsilon ^{-2\gamma -\frac{1}{2}} t,x)$ and $\gamma$ is a constant such that
\begin{align}
\label{gamma}
&{\rm there \  exist \  constants\  } \kappa > \kappa' > 1{\rm \ which \  satisfy} \nonumber\\
&\begin{cases}
(\kappa ' +\frac{21}{40} + \frac{\gamma}{10})\vee 2\kappa ' < \kappa < \gamma -\frac{C_f}{\mu},\\
1< \kappa ' <\frac{1}{20} + \frac{\gamma}{5}.
\end{cases}
\end{align}
Then there exist a.s. positive random variable $C(\omega) \in L^\infty (\Omega )$ and stochastic processes $\xi_t ^\varepsilon$ such that
\begin{align}
P( \| \bar{u} ^\varepsilon (t,\cdot )-\chi _{\xi^\varepsilon _t}(\cdot )\| _{L^2(\mathbb{R})} \leq \delta \  for \ all \  t \in [C(\omega )\varepsilon^{2\gamma +\frac{3}{2}} |\log \varepsilon |,T]  ) \to 1\ \ \ (\varepsilon \to 0), \nonumber
\end{align}
for all $\delta >0$ and $T>0$. Moreover, the distribution of the process $\xi^\varepsilon _t$ on $C([0,T], \mathbb{R})$ weakly converges to that of $\xi _t$ and $\xi_t$ obeys the SDE starting at $\xi_0$:
\begin{align}
 \label{eq:sde}
 d\xi_t = \alpha _1 a(\xi_t ) dB_t + \alpha _2 a(\xi_t)a'(\xi_t)dt,
\end{align}
where $\alpha _1$ and $\alpha _2 \in \mathbb{R}$ are defined as
\begin{align}
&\alpha _1 := \frac{1}{\| \nabla m \|_{L^2}} \nonumber \\
&\alpha _2 := - \frac{1}{\| \nabla m \|_{L^2} ^2} \int _0^\infty \int _{\mathbb{R}} \int _{\mathbb{R}} x p(t,x,y;m)^2 f''(m(y)) \nabla m(y) dxdydt \nonumber
\end{align}
and $p(t,x,y;m)$ denotes the fundamental solution for $\frac{\partial}{\partial t} - \Delta - f'(m)$ (See also \cite{f}, p. 252).
\end{theorem}

From the condition (\ref{gamma}), we need the condition $\gamma > \frac{19}{4}$ at least. This is the same condition as one of Funaki's result (see Theorem 8.1 in \cite{f}).

In this case, we can regard $C(\omega )\varepsilon^{2\gamma +\frac{3}{2}} |\log \varepsilon |$ as the first generation time in the time scale of order $O(\varepsilon^{-2\gamma -\frac{1}{2}})$. This is the same order as the first generation time for the deterministic case if we do not change the time scale. Our result covers that of \cite{f}. The time scale for the interface motion is the same.

Now we explain the idea of Funaki \cite{f94} and \cite{f} briefly. In \cite{f}, he showed that $\bar{u}^\varepsilon$ converges to $\chi _{\xi^\varepsilon _t}$ as $\varepsilon \to 0$, and the interface motion at the limit is described by (\ref{eq:sde}) in the case that the initial value $u_0 ^\varepsilon = m(\varepsilon ^{-\frac{1}{2}}(x-\xi _0 ))$. He took Ginzburg-Landau free energy as a Lyapnov functional corresponding to the equation (\ref{eq:pde}), which is defined by
\begin{align}
\mathcal{H}^\varepsilon (u) := \int _{\mathbb{R}} \left \{ \frac{1}{2}|\nabla u|^2 +\frac{1}{\varepsilon}F(u) \right \} dx \nonumber
\end{align}
where $f=-F'$. Note that the solution $u^\varepsilon$ of (\ref{eq:spde}) is not differentiable in $x$. Then, the set of minimizers of $\mathcal{H}^\varepsilon$ in the class of functions $u$ satisfying $u(\pm \infty)=\pm 1$ is given by $M^\varepsilon :=\{ m(\varepsilon ^{-\frac{1}{2}} (x-\eta)) | \eta \in \mathbb{R} \}$. Here we define a coordinate in the neighborhood of $M^1$ which is called Fermi coordinate. For $u\in \{u-m \in L^2(\mathbb{R}) \}$, we set $dist(u,M^1):=\inf _{\eta \in \mathbb{R}} \| u - m(\cdot -\eta) \| _{L^2(\mathbb{R})}$. If $dist (u,M^1)< \beta$ for some $\beta>0$, then there exists a unique constant $\eta(u) \in \mathbb{R}$ which attains $\inf _{\eta \in \mathbb{R}} \| u - m(\cdot -\eta) \| _{L^2(\mathbb{R})}$. And thus, we can see $u=m_{\eta(u)}+s(u)$ where $m_{\eta}(x)=m(x-\eta)$. We call the coordinate $(\eta(u),s(u)) \in \mathbb{R} \times L^2(\mathbb{R})$ Fermi coordinate.

If we change the time scale as $\bar{u}^\varepsilon(t,x):= u^\varepsilon(\varepsilon^{-2\gamma - \frac{1}{2}}t,x)$, $\bar{u}^\varepsilon$ satisfies an SPDE:
\begin{align}
\label{scalingspde}
\dot{\bar{u}}^\varepsilon = \varepsilon^{-2\gamma -\frac{1}{2}}\left \{ \Delta \bar{u}^\varepsilon +\frac{1}{\varepsilon}f(\bar{u}^\varepsilon)\right \} + (\varepsilon^{-2\gamma -\frac{1}{2}})^{\frac{1}{2}}\cdot \varepsilon^{\gamma} a(x) \dot{W}_t(x).
\end{align}
in a law sense. We give a formal proof of (\ref{scalingspde}). We have that
\begin{align}
\bar{u}^\varepsilon (t) - \bar{u}^\varepsilon (0) &= u^\varepsilon (\varepsilon^{-2\gamma - \frac{1}{2}} t) - u_0 ^\varepsilon \nonumber \\
&= \int _0 ^{\varepsilon^{-2\gamma -\frac{1}{2}}t} \left \{ \Delta u^\varepsilon (s) +\frac{1}{\varepsilon}f(u^\varepsilon (s))\right \} ds + \varepsilon^{\gamma} a(x) W_{\varepsilon^{-2\gamma -\frac{1}{2}}t} (x), \nonumber \\
&= \varepsilon^{-2\gamma -\frac{1}{2}} \int _0 ^t \left \{ \Delta \bar{u}^\varepsilon (s) +\frac{1}{\varepsilon}f(\bar{u}^\varepsilon (s))\right \} ds + (\varepsilon^{-2\gamma -\frac{1}{2}})^{\frac{1}{2}} \cdot \varepsilon^{\gamma} a(x) W_t (x), \nonumber
\end{align}
from SPDE (\ref{eq:spde}). In the third line, we have the first term from the integration by substitution $s \mapsto \varepsilon^{2\gamma +\frac{1}{2}}s$, and the second term comes from the self-similarity of space-time white noise (formally we have $W_{a^2 t}(x) = a W_t(x)$ in a law sense).

Because of the strong effect of the drift term, the solution of (\ref{eq:spde}) started from $m(\varepsilon ^{-\frac{1}{2}} (x-\xi _0)) \in M^\varepsilon$ should be attracted to $M^\varepsilon$. From this observation, Funaki \cite{f} showed that the solution $\bar{u}^\varepsilon$ did not go out of a tubular neighborhood of $M^\varepsilon$ in $L^2$-sense if the initial value was on $M^\varepsilon$, by investigating a structure of the functional $\mathcal{H}^\varepsilon$ around minimizers $M^\varepsilon$. And he derived an SDE as the dynamics of the interface by defining an appropriate coordinate on this neighborhood.

However, in our case, the initial value is not close to the neighborhood of $M^\varepsilon$. Thus, we need to show that the solution $u^\varepsilon$ enters the neighborhood of $M^\varepsilon$ in a short time with high probability, even if the initial value is not close to $M^\varepsilon$. We call this behavior the generation of interface.

We first prove the generation of interface in the deterministic case in Section \ref{sec2} as a preparation. We refer to the comparison argument in \cite{ham}. The proof of the main result is given in Section \ref{sec3}.

\section{The deterministic results}
\label{sec2}
In this section, we will show the generation of interface for the solution of PDE (\ref{eq:pde}). We assume that there exist positive constants $p$, $\mu>0$ such that the reaction term $f$ satisfies
\begin{align}
\label{reaction2}
 \begin{cases}
  \text{(i)}f \text{ has only three zeros }a_\pm, a_0, & \\
  \text{(ii)}f'( a_\pm) =-p < 0,\  f'(a_0)=\mu > 0, & \\
  \text{(iii)}f(u)\leq C(1+|u|^q)\text{ with some }C,q>0,& \\
  \text{(iv)}f'(u)\leq c\text{ with some }c>0, & \\
  \text{(v)}f(u)\leq -p(u-a_+) \ (u\geq a_+), & \\
  \text{(vi)}f(u)\geq -p(u-a_-) \ (u\leq a_-),
 \end{cases}
\end{align}
for some $a_-< a_0 < a_+$. We choose $a_\pm$ and $a_0$ because we need to change the stable points in order to construct super and sub solutions in Section \ref{sec3}. The initial value $u_0 ^\varepsilon$ satisfies the condition (\ref{eq:ini}) with $\pm 1$ replaced by $a_\pm$ throughout the rest of this section. We may take $C_0$ large enough such that $[a_- ,a_+]$ is included in $[-2C_0 , 2C_0]$. The argument in this section is based on Alfaro et al \cite{ham}. They proved that, for small $\eta >0$, the solution $u^\varepsilon$ formed an interface of width $O(\varepsilon ^{\frac{1}{2}})$ and each phase entered the $\eta$-neighborhood of $a_\pm$ uniformly at the time $t=\frac{1}{2\mu}\varepsilon | \log \varepsilon |$ (see Theorem 3.1 of \cite{ham}). However, in order to connect to the motion of interface, we need to show that the solution $u^\varepsilon$ enters $\varepsilon ^\kappa$-neighborhood of $M^\varepsilon$ in $L^2$-sense, that is $\inf_{\eta \in \mathbb{R}} \| u^\varepsilon - m(\varepsilon ^{-\frac{1}{2}}(\cdot -\eta)) \|_{L^2(\mathbb{R})} \leq \varepsilon ^\kappa$, for $\kappa >0$ and $\varepsilon ^\kappa \ll \eta$. And thus, we need to consider the time after $t=\frac{1}{2\mu}\varepsilon | \log \varepsilon |$.

\subsection{Auxiliary estimates}
\label{sec2-1}
We first prepare some preliminary results. We consider the ODE:
\begin{align}
\begin{cases}
\dot{Y} (\tau , \xi) = f( Y(\tau , \xi )),\ \ \  \tau >0,\\
Y(0, \xi)=\xi \in [-2C_0 ,2C_0].\end{cases} \nonumber
\end{align}

\begin{lemma}
\label{lem31}
There exists a constant $\eta_0 \in (0, a_+ - a_0)$ such that, for any $\eta \in (0,\eta_0 )$ and $\alpha >0$, there exists a positive constant $C>0$ and we have that
\begin{align}
Y(C|\log \varepsilon | ,\xi ) \geq a_+ -\eta \ \ \ (for\ all\ \xi \in [a_0+ \varepsilon ^\alpha ,a_+ -\eta]) \nonumber
\end{align}
for sufficiently small $\varepsilon >0$. The constant $C$ can be taken depending only on $\alpha$ and $f$.
\end{lemma}
\begin{proof}
First, we take $\eta_0 \in (0, a_+ - a_0)$ small enough and fix $\eta \in (0,\eta_0 )$. We explain about $\eta _0$ in the proof of next lemma. Since the solutions $Y(\tau,\xi)$ are larger than $Y(\tau,a_0 + \varepsilon ^\alpha)$ for all $\xi \in (a_0 + \varepsilon ^\alpha ,a_+ -\eta]$, the conclusion follows once we can show it for $Y(\tau,a_0 + \varepsilon ^\alpha)$. Corollary 3.5 in \cite{ham} implies that there exists a positive constant $C_1(\eta)>0$ such that
\begin{align}
C_1(\eta) e^{\mu \tau} \varepsilon ^\alpha \leq Y(\tau,a_0+ \varepsilon ^\alpha) -a_0 \nonumber
\end{align}
for $\tau >0$ where $Y(\tau,\varepsilon ^\alpha)$ remains in $(a_0 ,a_+ -\eta]$. An inequality $C_1(\eta) e^{\mu \tau} \varepsilon ^\alpha \geq a_+ -a_0 -\eta$ implies that
\begin{align}
\tau \geq \frac{\alpha}{\mu} |\log \varepsilon |+\frac{1}{\mu} \log \frac{a_+ -a_0 -\eta}{C_1(\eta)} . \nonumber
\end{align}
And thus, if we take $\frac{\tilde{\alpha}}{\mu}$ as the constant $C$ for small $\tilde{\alpha} >\alpha$ and take $\varepsilon >0$ sufficiently small, this lemma is proven.
\qed \end{proof}

\begin{lemma}
\label{lem32}
There exists a constant $\eta_0 \in (0, a_+ - a_0)$ such that, for any $\eta \in (0,\eta_0 )$ and $\kappa >0$, there exists a positive constant $C>0$ and we have that
\begin{align}
Y(C|\log \varepsilon | ,\xi )\geq a_+ -\varepsilon ^\kappa \ \ \ (for\ all\ \xi \in [a_+ -\eta , a_+ - \varepsilon ^\kappa]) \nonumber
\end{align}
for sufficiently small $\varepsilon >0$. The constant $C$ can be taken depending only on $\kappa$ and $f$.
\end{lemma}
\begin{proof}
From the same observation as the proof of Lemma \ref{lem31}, we only consider the solution $Y(\tau ,a_+ -\eta)$. We take small $\eta_0 \in (0, a_+ - a_0)$ such that the sign of the derivative $f''(u)$ does not change on $u\in [a_+ -\eta _0 ,a_+ )$, and fix $\eta \in (0,\eta_0 )$. At first, we consider the case that $f''(u) \leq 0$ on $[a_+ -\eta ,a_+)$. The inequality $f(u) \geq -\frac{f(a_+ -\eta)}{\eta}(u-a_+ )$ on $u\in [a_+ -\eta ,a_+ )$ and an easy computation gives us
\begin{align}
Y(\tau ,a_+ -\eta ) \geq a_+ -\eta e^{-\frac{f(a_+ -\eta)}{\eta} \tau} \nonumber
\end{align}
for all $\tau >0$. Reminding that $f'(a_+ )$ is negative and $f''(u) \leq 0$ on $\in [a_+ -\eta ,a_+ )$, we can show that the inequality
\begin{align}
\tau \geq \frac{1}{f'(a_+)}\{ \kappa |\log \varepsilon |+ \log \eta \} \geq -\frac{\eta}{f(a_+-\eta)}\{ \kappa |\log \varepsilon |+ \log \eta \}  \nonumber
\end{align}
implies that $a_+ -\eta e^{-\frac{f(a_+ -\eta)}{\eta} \tau} \geq a_+ -\varepsilon ^\kappa$. We can show this lemma by taking $C= \frac{\tilde{\kappa}}{p}$ for small $\tilde{\kappa} > \kappa$. The case that $f''(u) \geq 0$ on $\in [a_+ -\eta ,a_+)$ is easier than another one, because of the estimate $f(u) \geq -p(u-a_+)$ on $u\in [a_+-\eta ,a_+)$. The same argument as above gives us the same estimate and this completes the proof of this lemma.
\qed \end{proof}

Combining Lemma \ref{lem31} and \ref{lem32}, we obtain a useful estimate as following. We need this estimate when we connect the generation and motion of interface.

\begin{proposition}
\label{thm32}
For each $\alpha >0$ and $\kappa>0$, there exists a positive constant $C>0$ such that
\begin{align}
|Y(C_1 |\log \varepsilon | ,\xi )- a_+ |\leq \varepsilon ^\kappa \ \ \ for\ all\ \xi \in [a_0 + \varepsilon ^\alpha ,2C_0] \nonumber
\end{align}
for sufficiently small $\varepsilon >0$. The constant $C$ can be taken depending only on $\alpha$, $\kappa$ and $f$.
\end{proposition}
\begin{proof}
From the condition (v) of (\ref{reaction2}), we have that $f(u)\leq -p(u-a_+)$ on $u\in [a_+, 2C_0]$. And thus the similar argument to the proof of Lemma \ref{lem32} gives us
\begin{align}
Y(C|\log \varepsilon | ,\xi ) \leq a_+ +\varepsilon ^\kappa \ \ \ for\ all\ \xi \in [a_+ ,2C_0] \nonumber
\end{align}
if we take $C=\frac{\tilde{\kappa}}{p}$ for $\tilde{\kappa} > \kappa$. If we set $\tilde{\kappa} > \kappa$ and $\tilde{\alpha} > \alpha$, the solution $Y$ started from $[a_0+ \varepsilon ^\alpha ,a_+ -\eta]$ becomes larger than $a_+ -\eta$ until the time $t=\frac{\tilde{\alpha}}{\mu} |\log \varepsilon|$ from Lemma \ref{lem31}, and the solution started form $[a_0 - \eta ,2C_0]$ goes into $[a_0 - \varepsilon ^\kappa ,a_0 + \varepsilon ^\kappa]$ until the time $t=\frac{\tilde{\kappa}}{p} |\log \varepsilon|$ from Lemma \ref{lem32}. And thus, we can prove this proposition if we take $C_1 =\{ \frac{\tilde{\alpha}}{\mu} + \frac{\tilde{\kappa}}{p} \}$ for $\tilde{\kappa} > \kappa$ and $\tilde{\alpha} > \alpha$.
\qed \end{proof}

We can obtain the similar estimate to that of Proposition \ref{thm32} in the case that $\xi \in [-2C_0 ,a_0 -\varepsilon ^\alpha ]$. We state this below.

\begin{proposition}
\label{thm33}
For each $\alpha >0$ and $\kappa>0$, there exists a positive constant $C_1>0$ such that
\begin{align}
|Y(C_1 |\log \varepsilon | ,\xi ) -a_- |\leq \varepsilon ^\kappa \ \ \ for\ all\ \xi \in [-2C_0 ,a_0 -\varepsilon ^\alpha ] \nonumber
\end{align}
for sufficiently small $\varepsilon >0$. Especially the constant $C_1$ depends only on $\alpha$, $\kappa$ and $f$.
\end{proposition}

\subsection{Construction of super and sub solutions}
We set
\begin{align}
w_\varepsilon ^\pm (t,x) = Y \left( \frac{t}{\varepsilon} , u_0 ^\varepsilon (x) \pm \varepsilon  h(x) (e^\frac{\mu t}{\varepsilon }-1 ) \right) \nonumber
\end{align}
for a bounded positive function $h(x)\in C_b ^2 (\mathbb{R})$ which satisfies
\begin{align}
\label{condh}
\begin{cases}
{\rm (i)}\ \mu h \geq \{ u_0 ^{\varepsilon \prime} +\varepsilon h' (e^{\frac{\mu t}{\varepsilon}} -1) \} ^2\ for\ all\ t\in [0,C_1\varepsilon |\log \varepsilon|]\ and\ x\in \mathbb{R}, \\
{\rm (ii)}\ \mu h \geq \{ u _0 ^{\varepsilon \prime} +\varepsilon h' (e^{\frac{\mu t}{\varepsilon}} -1) \} ^2 + \{ \Delta u _0 ^{\varepsilon} + \varepsilon \Delta h (e^{\frac{\mu t}{\varepsilon}} -1) \}\\
\ \ \ \ \ \ \ \ \ \ \ \ \ \ \ \ \ \ \ \ \ \ \ \ \ \ \ \ \ \ \ \ \ \ \ \ \ \ \ \ for\ all\ t\in [0,C_1\varepsilon |\log \varepsilon|]\ and\ x\in \mathbb{R}, \\
{\rm (iii)}\ \varepsilon ^\kappa C_\mu \exp(-\frac{\sqrt{\mu} x}{2}) + h(x) ( \varepsilon ^{1-C_1 \mu}-\varepsilon ) \leq C \varepsilon^\kappa \exp (-\frac{\sqrt{\mu} x}{2}) \ for\ all\ x \geq K, \\
{\rm (iv)}\ \varepsilon ^\kappa C_\mu \exp(\frac{\sqrt{\mu} x}{2}) + h(x) ( \varepsilon ^{1-C_1 \mu}-\varepsilon ) \leq C \varepsilon^\kappa \exp(\frac{\sqrt{\mu} x}{2}) \ for\ all\ x \leq -K, \\
{\rm (v)}\ \lim _{\varepsilon \to 0} ( \varepsilon ^{1-C_1 \mu}-\varepsilon ) \| h \|_\infty =0 ,
\end{cases}
\end{align}
for some constants $C>0$, $K>1$ and $\varepsilon _0 >0$, and for any $\varepsilon \in (0,\varepsilon _0]$. The constant $\mu$ is introduced in (\ref{eq:ini}). We need to construct the function $h$ satisfying (\ref{condh}).

\begin{lemma}
\label{consth}
There exists a function $h \in C_b ^2 (\mathbb{R})$ which satisfies (\ref{condh}).
\end{lemma}
\begin{proof}
For simplicity, we set $a_\varepsilon =\varepsilon (e^{\frac{\mu t}{\varepsilon}} -1)$. Note that $0< a_\varepsilon < \varepsilon ^{1-C_1 \mu}-\varepsilon$ if $t\in [0, C_1 \varepsilon |\log \varepsilon |)$. We can divide the initial value as $u_0 ^\varepsilon =\tilde{u}_0 ^\varepsilon + \varepsilon ^\kappa g$ where supp$\tilde{u}_0 ^\varepsilon \subset [-1,1]$, $|g| + |g'| + |g''| \leq C_\mu \exp (- \frac{\sqrt{\mu} |x|}{2})$ and $\tilde{u}_0 ^\varepsilon$, $g \in C^2(\mathbb{R})$ (see (\ref{eq:ini})). Now we take $h=\varphi +\varepsilon ^\kappa \psi$ where $\varphi$, $\psi \in C_b ^2(\mathbb{R})$ and positive. 

At first, we construct $\varphi$ satisfying;
\begin{align}
\label{condphi1}
&\mu \varphi \geq 4( \tilde{u} _0 ^{\varepsilon \prime} ) ^2 + 4 ( a_\varepsilon \varphi ' ) ^2 ,\\
\label{condphi2}
&\mu \varphi \geq 4( \tilde{u} _0 ^{\varepsilon \prime} ) ^2 + 4 ( a_\varepsilon \varphi ' ) ^2 + \tilde{u} _0 ^{\varepsilon \prime \prime} +a_\varepsilon \varphi '' +\varepsilon ^{\bar{\kappa}} 1_{[-1,1]},
\end{align}
where $1_{[-1,1]}$ is an indicator function and $0 < \bar{\kappa} < \kappa$. We take a constant $K>1$ which does not depend on $\varepsilon$. We set $\varphi(x)= \exp(-\varepsilon ^{-\beta} (x+K))$ when $x<-K$ and $\varphi(x)= \exp(-\varepsilon ^{-\beta} (-x-K))$ when $x>K$, for some $0<\beta <\frac{1-C_1\mu}{2}$. By using conditions $\tilde{u} _0 ^{\varepsilon \prime} =\tilde{u} _0 ^{\varepsilon \prime \prime} =0$ for $|x|>K$ and $a_\varepsilon \varepsilon ^{-\beta} \to 0$ as $\varepsilon \to 0$, the estimates (\ref{condphi1}) and (\ref{condphi2}) are established when $|x|>K$. We can take a constant which is larger than $\frac{4 C_0 ^2 + C_0}{\mu}$ as $\varphi$ on $[-1,1]$. Thus we need to connect $\varphi$ on $[-K ,-1]$ and $[1,K]$. We consider only on $[-K ,-1]$. We take $\varphi ''$ as an linear function on $[-K ,-K + \varepsilon ^{2\beta} ]$ where $\varphi ''(-K + \varepsilon ^{2\beta} ) =0$. Then $\varphi$ becomes a monotonous increasing function on $[-K ,-K + \varepsilon ^{2\beta} ]$, and $\varphi (-K + \varepsilon ^{2\beta} ) =1 + C \varepsilon ^{\beta}$ for some $C>0$. Next, we take a concave function $\varphi(x)$ on $[-K + \varepsilon ^{2\beta} , -1]$ where $\varphi (-1) > \frac{4 C_0 ^2 + C_0}{\mu}$ and $\varphi '(-1)=\varphi ''(-1)=0$ and $\varphi$ is twice differentiable at $x= -K + \varepsilon ^{2\beta}$. For example, we can take $\varphi ''$ equals to some negative constant on $[-K + \varepsilon ^{2\beta} + \delta , -1-\delta]$ for some small $\delta>0$, interpolate on $[-K + \varepsilon ^{2\beta}  , -1] \backslash [-K + \varepsilon ^{2\beta} + \delta , -1-\delta]$ by linear functions and integrate them to get $\varphi$ which satisfies conditions as above. Because $\varphi '' \leq 0$, $\varphi$ is concave, and $\varphi '(x) \leq O(\varepsilon ^\beta)$. Note that $\varphi (-1 )= O(\varepsilon ^{-\beta})$. Combining this and the conditions $a_\varepsilon \varepsilon ^{-\beta} \to 0$ and $\tilde{u} _0 ^{\varepsilon \prime} =\tilde{u} _0 ^{\varepsilon \prime \prime}=0$, we can prove (\ref{condphi1}) and (\ref{condphi2}). We take $\varphi$ symmetrically on $[1,K]$.

Next we construct $\psi$ satisfying;
\begin{align}
\label{condpsi1}
&\varepsilon ^\kappa \mu \psi \geq 4( \varepsilon ^\kappa g' )^2 + 4 ( \varepsilon ^\kappa a_\varepsilon \psi ' ) ^2 ,\\
\label{condpsi2}
&\varepsilon ^\kappa \mu \psi \geq 4( \varepsilon ^\kappa g' )^2 + 4 ( \varepsilon ^\kappa a_\varepsilon \psi ' ) ^2 + \varepsilon ^\kappa \psi '' + \varepsilon ^\kappa a_\varepsilon \psi ''.
\end{align}
When $|x|>1$, we take $\psi (x) = \exp (- \frac{\sqrt{\mu} |x|}{2})$. Then the right-hand side of (\ref{condpsi1}) and the sum of the first, second and fourth terms of the right-hand side of (\ref{condpsi2}) are smaller than $\varepsilon ^\kappa \mu \psi$ for sufficiently small $\varepsilon >0$. The third term of (\ref{condpsi2}) is $\frac{\varepsilon ^\kappa \mu}{4} \psi$. This term is smaller than $\varepsilon ^\kappa \mu \psi $ and larger than $\varepsilon ^\kappa g''$ from the definition of $C_\mu$. We use the condition (iv) and (v) of (\ref{eq:ini}) here. Let us discuss about $|x| \leq 1$. We take $\psi \in C^2(\mathbb{R})$ which is twice differentiable at $x=-1$, $\psi '' (-1+\delta) = 0$ for some $\delta \in (0,1)$ and $\psi ''$ is monotonous decreasing when $x \in [-1,-1+\delta]$. For example, we can take cubic function because we have four conditions for the values of $\psi$, $\psi '$ and $\psi ''$ at $x=-1$ and $x=-1+\delta$. In particular, $\psi$ is positive on $[-1,-1+\delta]$. We can take $\psi$ similarly and symmetrically on $[1-\delta , 1]$. We connect $\psi$ by a concave function on $[-1+\delta , 1-\delta]$ which is twice differentiable at $x=-1+\delta$, $1-\delta$. For example, we can take a quartic function $\psi (x) = ax^4 +bx^2 + c$ for certain $a,b,c \in \mathbb{R}$, because we have six conditions of $\psi$, $\psi '$ and $\psi ''$ and take symmetric function $\psi$. And $\psi$ is positive on $\mathbb{R}$. In a similar way as above, we can show (\ref{condpsi1}) and (\ref{condpsi2}) on $[-1,1]$. We use the conditions $\mu \psi (\pm 1) > \psi '' (\pm1)$, the monotonous decreasing (resp. increasing) of $\psi ''$ on $[-1, -1+\delta]$ (resp. $[1-\delta ,1]$) and $\psi '' <0$ on $[-1+\delta ,1-\delta]$.

To sum up (\ref{condphi1}) and (\ref{condpsi1}), we have that
\begin{align}
\mu h \geq 2(u_0 ^{\varepsilon \prime})^2 + 2 ( \varepsilon ^\kappa a_\varepsilon h ' ) ^2 \geq (u_0 ^{\varepsilon \prime} + \varepsilon ^\kappa a_\varepsilon h ' )^2. \nonumber
\end{align}
Here we use $(a+b)^2 \leq 2a^2 +2b^2$ twice. Similarly, we get
\begin{align}
\mu h \geq \{ u _0 ^{\varepsilon \prime}+a_\varepsilon h' \} ^2 + \{ u_0 ^{\varepsilon \prime \prime} + a_\varepsilon h'' \}, \nonumber
\end{align}
from (\ref{condphi2}) and (\ref{condpsi2}). Note that $\psi '' + \varepsilon ^{\bar{\kappa} -\kappa} 1_{[-1,1]} > g''$ for sufficiently small $\varepsilon$ because of the constant $C_\mu$ and the boundedness of $g''$ on $[-1,1]$. We also use the estimate $\psi '' > g''$ when $|x|>1$, which we show in the previous paragraph. Here we note $\varphi$ and $\varepsilon ^\kappa \psi$ depend on $\varepsilon$, however they are bounded by some constant which diverges in the order $O(\varepsilon ^{-\beta})$ and larger than $4C_0 ^2 +C_0$ for sufficiently small $\varepsilon >0$, from the construction of these function. This and the convergence $a_\varepsilon \varepsilon ^{-\beta} \to 0$ show the condition (v) of (\ref{condh}). We can see (iii) and (iv) from the construction of $\varphi$ and $\psi$.
\qed \end{proof}

Now we prove that $w_\varepsilon ^\pm$ are super and sub solutions for (\ref{eq:pde}) by applying the maximum principle. See (\ref{condh}) for the precise condition of $h$. Our claim in this subsection is formulated in the following proposition.
\begin{proposition}
\label{thm61}
If we fix a constant $0 < C_1 < \frac{1}{\mu}$ and a positive function $h(x)\in C_b ^2 (\mathbb{R})$ which satisfies (\ref{condh}), then there exists $\varepsilon _0 >0$ such that for all $\varepsilon \in (0,\varepsilon_0]$, $t \in [0, C_1 \varepsilon |\log \varepsilon |)$ and $x \in \mathbb{R}$, we have that $w_{\varepsilon} ^- (t,x) \leq u^{\varepsilon} (t,x) \leq w_{\varepsilon } ^+ (t,x)$ where $u^{\varepsilon}$ is the solution of (\ref{eq:pde}).
\end{proposition}
Before proving this proposition, we give a notation as a preparation in advance. For $\xi \neq a_\pm,$ $a_0$, we define the following function:
\begin{align}
\label{def5-1}
A (\tau ,\xi ) := \frac{Y_{\xi \xi }  (\tau ,\xi )}{Y_\xi  (\tau ,\xi )},
\end{align}
where $Y_\xi$ and $Y_{\xi \xi}$ mean the derivatives of $Y$ with respect to $\xi$. We get an ODE:
\begin{align}
\label{ode6-1}
\begin{cases}
{Y}_{\xi \tau}  (\tau ,\xi ) = {Y}_\xi  (\tau ,\xi )  f'( Y  (\tau ,\xi ) ),\ \ \ \tau > 0,\\
Y _\xi (0, \xi )=1,
\end{cases}
\end{align}
and we obtain
\begin{align}
\label{eq5-1}
{Y}_\xi  (\tau ,\xi ) = \exp \left( \int _0 ^\tau f'\left( Y (\tau ,\xi )\right) ds \right),\ \ \ \tau \geq0,
\end{align}
from (\ref{ode6-1}). In particular, ${Y}_\xi$ is positive and thus we can define $A  (\tau ,\xi )$ as (\ref{def5-1}). We get 
\begin{align}
A (\tau, \xi ) = \int _0 ^\tau Y_\xi (s,\xi ) f''\left( Y(s,\xi )\right) ds,\ \ \ \tau \geq 0, \nonumber
\end{align}
by computing $Y_{\xi \xi}$ from (\ref{eq5-1}). Now we prove Proposition \ref{thm61} by using the maximum principle.

\begin{proof}[Proof of Proposition \ref{thm61}]
We fix $0 < C_1 < \frac{1}{\mu}$. At first, we need to check the initial conditions $\xi$ in $w_\varepsilon ^\pm$ are in $[-2C_0,2C_0]$. When $t \in [0,C_1 \varepsilon | \log \varepsilon |]$ and $\varepsilon$ is sufficiently small, we have that
\begin{align}
u_0 ^{\varepsilon} + \varepsilon h(x) ( e^{\frac{\mu t}{\varepsilon}} -1 ) \leq C_0 + h(x) (\varepsilon ^{1-C_1 \mu} - \varepsilon) \leq 2C_0 \nonumber
\end{align}
where $h$ is taken in Lemma \ref{consth} and $C_1< \frac{1}{\mu}$. Here we use the condition (v) of (\ref{condh}). In the same way, we can estimate $u_0 ^{\varepsilon} - \varepsilon h(x) ( e^{\frac{\mu t}{\varepsilon}} -1 )\geq -2C_0$. Let $\mathcal{L}$ be an operator which is defined by
\begin{align}
\mathcal{L} (u)(t,x) := \dot{u}(t,x) - \Delta u(t,x) - \frac{1}{\varepsilon}f(u(t,x)). \nonumber
\end{align}
From the maximum principle, if $\mathcal{L} (w_{\varepsilon} ^+)\geq 0$ then $w_{\varepsilon } ^+\geq u^{\varepsilon }$ (see Theorem 9 in \cite{fr} and proof of Lemma 2.2 in \cite{f} for the comparison of solutions from the maximum principle). A direct computation gives us
\begin{align}
\mathcal{L} (w_{\varepsilon} ^+)(t&,x) = \frac{1}{\varepsilon} Y_\tau + \mu h  e^{\frac{\mu t}{\varepsilon}} Y_\xi - \{ \Delta u_0 ^{\varepsilon} +\varepsilon \Delta h (e^{\frac{\mu t}{\varepsilon}} -1) \} Y_\xi \nonumber \\
&\ \ \ \ \ \ \ \ \ \ \ \ \ \ \ \ \ \ \ \ \ \ \ \ \ \ \ \ \ \ \ \ \ \ \ \ \ \ \ \ \ \ \ - \{ u_0 ^{\varepsilon \prime} +\varepsilon h' (e^{\frac{\mu t}{\varepsilon}} -1) \} ^2 Y_{\xi \xi} - \frac{1}{\varepsilon}f(Y )\nonumber \\
&= \mu h  e^{\frac{\mu t}{\varepsilon}} Y_\xi - \{ \Delta u_0 ^{\varepsilon } +\varepsilon \Delta h (e^{\frac{\mu t}{\varepsilon}} -1) \} Y_\xi - \{ u _0 ^{\varepsilon \prime} +\varepsilon h' (e^{\frac{\mu t}{\varepsilon}} -1) \} ^2 Y_{\xi \xi} \nonumber \\
&= \left[ \mu h  e^{\frac{\mu t}{\varepsilon}} - \{ \Delta u_0 ^{\varepsilon } +\varepsilon \Delta h (e^{\frac{\mu t}{\varepsilon}} -1) \} - \{ u_0 ^{\varepsilon \prime} +\varepsilon h' (e^{\frac{\mu t}{\varepsilon}} -1) \} ^2 A \right ] Y_{\xi} \nonumber \\
&\geq \left[ \left \{ \mu h - \{ u_0 ^{\varepsilon \prime} +\varepsilon h' (e^{\frac{\mu t}{\varepsilon}} -1) \} ^2 \right \}  e^{\frac{\mu t}{\varepsilon}} - \{ \Delta u_0 ^{\varepsilon } +\varepsilon \Delta h (e^{\frac{\mu t}{\varepsilon}} -1) \} \right] Y_{\xi} \nonumber \\
&\geq \left [ \mu h - \{ u_0 ^{\varepsilon \prime} +\varepsilon h' (e^{\frac{\mu t}{\varepsilon}} -1) \} ^2 - \{ \Delta u_0 ^{\varepsilon } +\varepsilon \Delta h (e^{\frac{\mu t}{\varepsilon}} -1) \} \right ] Y_{\xi} \nonumber
\end{align}
for all $x \in \mathbb{R}$ and $t\in[0,C_1 \varepsilon |\log \varepsilon |]$. Note that the function $Y_{\xi} $ is positive. The definition of $A$ gives us the third equality. The fourth inequality comes from Lemma 3.7 of \cite{ham} and the fifth inequality comes from the condition (i) of (\ref{condh}). From (ii) of (\ref{condh}), we see that $\mathcal{L} (w_{\varepsilon} ^+)\geq 0$. So we have proved that $w_{\varepsilon } ^+ \geq u^{\varepsilon }$ holds for all $t \in [0, C_1 \varepsilon |\log \varepsilon |]$ and $x \in \mathbb{R}$. The converse $w_{\varepsilon} ^- \leq u^{\varepsilon}$ can be proved in a similar way.
\qed \end{proof}

\subsection{The generation of interface in the deterministic case}
Now we formulate and prove the conclusion of this section.
\begin{proposition}
\label{pro21}
If $u^\varepsilon$ is the solution of PDE (\ref{eq:pde}) and $\mu$ is defined in (\ref{reaction}), then there exist $K>1$, $\kappa > \frac{1}{2}$, $C>0$ and $\widetilde{C}>0$ such that for sufficiently large  $C_1 \in (0, \frac{1}{\mu})$ and any $0<\bar{\beta}< 1-C_1 \mu$.\\
{\rm (i)} $a_- - \varepsilon ^\kappa \leq u ^\varepsilon (C_1 \varepsilon |\log \varepsilon |,x) \leq a_+ +\varepsilon ^\kappa \ (x \in [-K,K])$, \\
{\rm (ii)} $u ^\varepsilon (C_1 \varepsilon |\log \varepsilon |,x) \geq a_+ -\varepsilon ^\kappa$ (for all $x \in [-K,K]$ such that $u_0 ^{\varepsilon}(x) \geq a_0 + C \varepsilon ^{\bar{\beta}}$),

$u ^\varepsilon (C_1 \varepsilon |\log \varepsilon |,x) \leq a_- +\varepsilon ^\kappa$  (for all $x \in [-K,K]$ such that $u_0 ^{\varepsilon}(x) \leq a_0 - C \varepsilon ^{\bar{\beta}}$),\\
{\rm (iii)} $| u ^\varepsilon (C_1 \varepsilon |\log \varepsilon |,x) - a_+ | \leq \widetilde{C} \varepsilon ^\kappa \exp (-\frac{\sqrt{\mu} x}{2}) \  (x\geq K )$,

\ $| u ^\varepsilon (C_1 \varepsilon |\log \varepsilon |,x) - a_- | \leq \widetilde{C} \varepsilon ^\kappa \exp (\frac{\sqrt{\mu} x}{2}) \  (x\leq  -K)$.
\end{proposition}
\begin{proof}
(i) Propositions \ref{thm32} and \ref{thm61} imply that there exists a constant $C_1>0$ such that 
\begin{align}
u^\varepsilon (C_1\varepsilon |\log \varepsilon|,x) &\leq w_\varepsilon ^+ (C_1\varepsilon |\log \varepsilon|,x) \nonumber \\
&\leq Y \left( C_1 |\log \varepsilon| , u_0 ^{\varepsilon} (x) + h(x) ( \varepsilon ^{1-C_1 \mu}-\varepsilon ) \right) \leq a_+ + \varepsilon ^\kappa . \nonumber
\end{align}
Remind that the estimate $|u_0 ^{\varepsilon} (x) + h(x) ( \varepsilon ^{1-C_1 \mu}-\varepsilon )| \leq 2C_0 \nonumber$ holds for all $x \in \mathbb{R}$, $C_1\in (0,\frac{1}{\mu})$ and sufficiently small $\varepsilon >0$. The proof of the lower bound is similar.\\
(ii) We only show the first estimate. From Proposition \ref{thm61}, we obtain
\begin{align}
u^\varepsilon (C_1\varepsilon |\log \varepsilon| & ,x) \geq w_\varepsilon ^- (C_1\varepsilon |\log \varepsilon| ,x) = Y \left( C_1 |\log \varepsilon| , u_0 ^{\varepsilon} (x) - h(x) ( \varepsilon ^{1-C_1 \mu}-\varepsilon ) \right). \nonumber
\end{align}
Here we need to observe the neighborhood of $\xi_0$ which is the zero of $u_0 ^{\varepsilon}$. The condition $u_0 ^{\varepsilon} (x) - h(x) ( \varepsilon ^{1-C_1 \mu}-\varepsilon ) \geq \varepsilon^\alpha$ is equivalent to
\begin{align}
\label{est7-1}
u_0 ^{\varepsilon} (x) \geq h(x) ( \varepsilon ^{1-C_1 \mu}-\varepsilon ) + \varepsilon^\alpha ,
\end{align}
for $\alpha >0$. Thus there exists a positive constant $C>0$, and $u_0 ^{\varepsilon} (x) \geq C\varepsilon ^{\bar{\beta}}$ implies (\ref{est7-1}) if $\alpha >\bar{\beta}$ by taking $2\beta =1-C_1 \mu -\bar{\beta}$ in the construction of $h$ (see the proof of Proposition \ref{consth}).\\
(iii) We only show the first. From the definition of $h \in C_b ^2 (\mathbb{R})$, we immediately see that
\begin{align}
u^\varepsilon (C_1\varepsilon |\log \varepsilon| ,x) -a_+ &\leq w_\varepsilon ^+ (C_1\varepsilon |\log \varepsilon| ,x) - a_+ \nonumber \\
&=Y \left( C_1 |\log \varepsilon| , u_0 ^{\varepsilon} (x) + h(x) ( \varepsilon ^{1-C_1 \mu}-\varepsilon ) \right)-a_+\nonumber \\
&\leq u_0 ^{\varepsilon} (x) + h(x) ( \varepsilon ^{1-C_1 \mu}-\varepsilon )-a_+ \leq C \varepsilon^\kappa \exp \left ( -\frac{\sqrt{\mu} x}{2}\right ), \nonumber
\end{align}
for all $x>K$ from the condition (iii) of (\ref{condh}). The first inequality comes from Proposition \ref{thm61}. We also have that
\begin{align}
u^\varepsilon (&C_1\varepsilon |\log \varepsilon| ,x) -a_+ \geq - C \varepsilon^\kappa \exp \left ( -\frac{\sqrt{\mu} x}{2}\right ) . \nonumber
\end{align}
We get $|u^\varepsilon (C_1\varepsilon |\log \varepsilon| ,x) -a_- | \leq C \varepsilon^\kappa \exp (\frac{\sqrt{\mu} x}{2})$ for $x<-K$ in a similar way.
\qed \end{proof}

\section{Proof of Theorem {\ref{thm23}}}
\label{sec3}
In this section, we consider the SPDE (\ref{eq:spde}). Recall that the external noise term $\dot{W}^\varepsilon _t (x)$ is given by $\varepsilon ^\gamma a(x) \dot{W} _t (x)$ where $\dot{W} _t (x)$ is a space-time white noise and $a\in C_0^\infty ([-1,1])$, and that the reaction term $f$ satisfies (\ref{reaction}) and the initial value $u^\varepsilon _0$ satisfies (\ref{eq:ini}). Throughout this section, we set constants $C_f$ and $\kappa ' >1$, and assume that the constants $C_1$, $\alpha$ and $\kappa$ satisfying
\begin{align}
 \label{const}
  \begin{cases}
  C_f := \underset{u \in [-2C_0,2C_0]}{\sup} f'(u),& \\
  \alpha >\frac{1}{2},\ \kappa > \kappa ' >1, & \\
  \frac{\alpha}{\mu} + \frac{\kappa}{p} +\bar{\delta} \leq C_1 \leq \frac{1}{\mu}-\bar{\delta},
  \end{cases}
\end{align}
for sufficiently small $\bar{\delta}>0$. The constants $p$ and $\mu$ are introduced in (\ref{reaction}), and $C_0$ is introduced in (\ref{eq:ini}). In particular, the constant $C_1>0$ is the same constant as in Proposition \ref{pro21}.

\subsection{Preliminary results}
At the beginning of this section, we refer to the result about a property of the solution $u^\varepsilon$; see Section 2 of \cite{f} or Theorem 3.1 of \cite{f94}.
\begin{proposition}
\label{thm31}
If $|u^\varepsilon _0 (x)|\leq K$, then we have that
\begin{align}
\lim_{\varepsilon \to 0} P\left ( |u^\varepsilon  (t,x)| \leq \max \{K,1\} +\delta \ for\ all\ t \in [0,\varepsilon ^{-n}]\ and\ all\ x\in \mathbb{R}\right ) =1, \nonumber
\end{align}
for all $n$ and $\delta >0$.
\end{proposition}
From this result, we see that the solution $u^\varepsilon$ stays in the interval $[-2C_0 ,2C_0]$ up to high probability. We introduce a stopping time
\begin{align}
\tau _1 :=\inf \{t >0 |  |u^\varepsilon  (t,x)| > 2C_0\ for\ some\ x\in \mathbb{R} \}, \nonumber
\end{align}
so that $u^\varepsilon$ stays in $[-2C_0 ,2C_0]$ until the time $\tau _1$. The probability that $\tau _1 \geq \varepsilon ^{-n}$ occurs tends to 1 as $\varepsilon \to 0$ for each $n\in \mathbb{N}$ from Proposition \ref{thm31}.

Next we prove that the solution $u^\varepsilon$ of (\ref{eq:spde}) is close to a solution $u$ of (\ref{eq:pde}) if the initial values are same. The proof is based on the proof of Proposition 12.1 of \cite{dpz}. As a preparation, we show an estimate for a stochastic convolution. Note that we apply this result for small $T$ later.

\begin{lemma}
\label{lem87}
For each $p>4$ and $a(x)\in C_0 ^\infty (\mathbb{R})$, there exists a positive constant $C_{a,p}>0$ such that
\begin{align}
E\left [ \underset{t\in[0,T]}{\sup}\left \| \int _0^t \langle S_{t-s} a(\cdot), dW_s (\cdot )\rangle \right \| _{L^2(\mathbb{R})}^p\right ] \leq C_{a,p} T, \nonumber
\end{align}
holds for every $0<T<1$.
\end{lemma}
\begin{proof}
We use the factorization method (see Proposition 5.9, Theorem 5.10 and Proposition 7.3 of \cite{dpz}). Note that $\| S_{t}a \| _{HS} ^2 \leq C t^{-\frac{1}{2}} \| a (\cdot ) \| _{L^2(\mathbb{R})} ^2$ for some constant $C>0$, where $\| \cdot \| _{HS}$ is a Hilbert-Schmidt norm on $L^2(\mathbb{R})$ and $a:L^2(\mathbb{R}) \to L^2(\mathbb{R})$ is a multiplication operator which is defined by $(af)(x):=a(x)f(x)$. Indeed, from Chapman-Kolmogolov equation, we have that $\| S_{t}a \| _{HS} ^2 = \int _\mathbb{R} \int _\mathbb{R} p(t,x,y)^2 a(y)^2dxdy = \int _\mathbb{R} p(2t,y,y) a(y)^2dy = \frac{1}{\sqrt{4\pi t}} \| a (\cdot ) \| _{L^2(\mathbb{R})} ^2$. From this observation and the stochastic Fubini's theorem, we have that
\begin{align}
\int _0^t \langle S_{t-s} a(\cdot), dW_s (\cdot )\rangle = \frac{\sin \pi \alpha}{\pi} \int _0 ^t (t-s)^{\alpha -1} S_{t-s} Y(s)ds, \nonumber
\end{align}
where
\begin{align}
Y(s)=\int _0 ^s (s-r )^{-\alpha} \langle S_{s-r} a(\cdot), dW_r (\cdot )\rangle \nonumber
\end{align}
for each $\alpha \in (\frac{1}{p} , \frac{1}{4})$. Taking $q>1$ such that $\frac{1}{p}+\frac{1}{q} = 1$, we get
\begin{align}
\underset{t\in [0,T]}{\sup} \left \| \int _0^t \langle S_{t-s} a(\cdot), dW_s (\cdot )\rangle \right \| _{ L^2(\mathbb{R})} ^p & \leq \left ( \frac{\sin \pi \alpha}{\pi}\right ) ^p \underset{t\in [0,T]}{\sup} \left ( \int _0 ^t |t-s|^{\alpha -1} \|Y(s)\|_{ L^2(\mathbb{R})} ds \right )^p \nonumber \\
& \leq C \underset{t\in [0,T]}{\sup} \left ( \int _0 ^t |t-s|^{q(\alpha -1)} ds \right )^{\frac{q}{p}} \cdot \left ( \int _0 ^T \|Y(s)\|_{ L^2(\mathbb{R})} ^p ds \right )  \nonumber \\
& \leq C \int _0 ^T \|Y(s)\|_{ L^2(\mathbb{R})} ^p ds  \nonumber
\end{align}
from H\"{o}lder's inequality, because $q(\alpha -1) > -1$ and $T \in (0,1)$. Next we derive an estimate for $Y(s)$:
\begin{align}
E\left [ \|Y(s)\| _{L^2(\mathbb{R})} ^p\right ] & \leq E\left [ \underset{s' \in [0,s]}{\sup}\left \| \int _0 ^{s'} (s-r)^{-\alpha} \langle S_{s-r} a(\cdot), dW_r (\cdot )\rangle \right \| _{L^2(\mathbb{R})} ^p\right ] \nonumber \\
& \leq C_p \left ( \int _0 ^{s} (s-r )^{-2\alpha} \left \| S_{s-r} a\right \| _{HS} ^2 dr \right )^{\frac{p}{2}} \leq C_{a,p} s^{\frac{p}{2}(\frac{1}{2} -2\alpha)}. \nonumber
\end{align}
We have used Burkholder's inequality in the second line. To sum up these estimates, we can show this lemma noting $\frac{1}{2} -2\alpha > 0$.
\qed \end{proof}

\begin{proposition}
\label{thm81}
Let $u(t,x)$ be a solution of PDE (\ref{eq:pde}) where $f$ satisfies (\ref{reaction}) and the initial value $u^\varepsilon _0$ satisfies (\ref{eq:ini}). Then, we have that 
\begin{align}
\lim _{\varepsilon \downarrow 0}P\left( \underset{t\in [0, \frac{ \varepsilon}{\mu} |\log \varepsilon | ]}{\sup} \| u^\varepsilon (t,\cdot ) - u(t,\cdot) \|_{L^2(\mathbb{R})} \leq \varepsilon ^{\kappa } \right) =1, \nonumber
\end{align}
where $\kappa < \gamma -\frac{C_f}{\mu}$.
\end{proposition}
\begin{proof}
At first, we consider the mild form
\begin{align}
u^\varepsilon(t) -u(t)=\frac{1}{\varepsilon} \int _0^t S_{t-s}\{ f(u^\varepsilon (s)) -f(u (s)) \} ds +u_1(t) , \nonumber
\end{align}
where $u_1(t):=\varepsilon ^\gamma \int _0^t \langle S_{t-s}a(\cdot ),dW_s(\cdot) \rangle$. We now consider stopping times $\sigma := \inf \{ t>0 | \| u^\varepsilon(t) -u(t) \|_{L^2} > \varepsilon ^{\kappa } \}$ and $\tau_2 := \tau _1\wedge \sigma$. From Proposition \ref{thm31} and the definition of the positive constant $C_f >0$ given in (\ref{const}), we obtain
\begin{align}
\| u^\varepsilon(t \wedge \tau_2) -u(&t \wedge \tau_2) \|_{L^2} \leq \frac{C_f}{\varepsilon} \int _0^{t \wedge \tau_2} \| u^\varepsilon (s) -u (s) \|_{L^2} ds + \| u_1(t \wedge \tau _2) \|_{L^2} \nonumber \\
&\leq \frac{C_f}{\varepsilon} \int _0^t \| u^\varepsilon (s \wedge \tau_2) -u (s \wedge \tau_2) \|_{L^2} ds + \| u_1(t \wedge \tau _2) \|_{L^2}. \nonumber
\end{align}
From Gronwall's inequality, we have that
\begin{align}
\| u^\varepsilon(t \wedge \tau_2) -u(t \wedge \tau_2) \|_{L^2} & \leq \varepsilon ^\gamma \exp \left (\frac{C_f T}{\varepsilon} \right) \cdot \underset{t\in [0, T]}{\sup}\| \varepsilon ^{-\gamma} u_1(t \wedge \tau _2) \|_{L^2}\nonumber \\
& \leq \varepsilon ^\gamma \exp \left (\frac{C_f T}{\varepsilon} \right) \cdot \underset{t\in [0, T]}{\sup}\|\varepsilon ^{-\gamma} u_1(t) \|_{L^2} \nonumber
\end{align}
for each $T>0$. From the estimate in Lemma \ref{lem87}, we obtain
\begin{align}
E[\| u^\varepsilon(T \wedge \tau_2) -u(T \wedge \tau_2) \|_{L^2} ^p] \leq C_p \varepsilon ^{p \gamma} \exp \left (\frac{pC_f T}{\varepsilon} \right) T  \nonumber
\end{align}
for every $p>4$ and $0<T<1$. As a result, for sufficiently large $p$, we obtain
\begin{align}
P(\sigma \leq \frac{ \varepsilon}{\mu} |\log \varepsilon |) &\leq P(\tau _2 \leq \frac{ \varepsilon}{\mu} |\log \varepsilon |) \leq \varepsilon ^{-p\kappa } E[\| u^\varepsilon(\frac{ \varepsilon}{\mu} |\log \varepsilon | \wedge \tau_2) -u(\frac{ \varepsilon}{\mu} |\log \varepsilon | \wedge \tau_2) \|_{L^2} ^p]\nonumber \\
&\leq C \varepsilon ^{p(\gamma -\kappa -\frac{C_f}{\mu}) +1} |\log \varepsilon | \nonumber
\end{align}
from Chebyshev inequality with the choice of $T=\frac{ \varepsilon}{\mu} |\log \varepsilon |$. Note that the right hand side converges to 0 as $\varepsilon \to 0$ in the order of $O(\varepsilon ^{p(\gamma -\kappa -\frac{C_f}{\mu}) +1} |\log \varepsilon |)$. From the conditions of $\gamma$ and $\kappa$, $p(\gamma -\kappa -\frac{C_f}{\mu}) +1$ is strictly positive. This estimate implies the conclusion.
\qed \end{proof}

Next we need to modify Proposition \ref{thm31}. The outline of the proof is similar to that of Theorem 2.1 in \cite{f}. At first we consider a stochastic process $u_1(t)$ in the proof of Proposition \ref{thm81}. Here $u_1$ satisfies the stochastic heat equation;
\begin{align}
\begin{cases}
\dot{u}_1 (t,x) &= \Delta u_1 (t,x)+\varepsilon ^\gamma a(x) \dot{W}_t(x) ,\ \ \ t > 0,\  x\in \mathbb{R}, \nonumber \\
u_1 (0,x) &= 0,\ \ \ x\in \mathbb{R}. \nonumber
\end{cases}
\end{align}
Now we refer to a result which asserts that the perturbation of the noise is very small; see Lemma 2.1 in \cite{f}.
\begin{lemma}
\label{lem81}
There exists a random variable $Y(\omega) \in \cap _{p\geq 1} L^p (\Omega )$ such that
\begin{align}
|u_1 (t,x)| \leq \varepsilon ^\gamma Y(\omega), \ \ \ t\in [0,1],\ x \in \mathbb{R},\ 0<\varepsilon <1. \nonumber
\end{align}
\end{lemma}

Next we consider a PDE;
\begin{align}
\begin{cases}
\dot{\bar{u}}_\pm ^{\varepsilon , \delta} (t,x) &= \Delta \bar{u}_\pm ^{\varepsilon , \delta} (t,x)+\displaystyle{\frac{1}{\varepsilon}} f_\pm ^\delta( \bar{u}_\pm ^{\varepsilon , \delta} (t,x) ) ,\ \ \ t > 0,\  x\in \mathbb{R}, \nonumber \\
\bar{u}_\pm ^{\varepsilon , \delta} (0,x) &= u_0 ^\varepsilon (x)\pm \delta ,\ \ \ x\in \mathbb{R}, \nonumber
\end{cases}
\end{align}
for small $\delta >0$, where the functions $f_\pm ^\delta \in C^2(\mathbb{R})$ satisfy the following conditions;
\begin{align}
f_+ ^\delta(u) \geq \sup _{|v|\leq \delta} f_+ ^\delta(u+v),\ f_+ ^\delta(\pm 1+\delta) =0,\ f_+ ^\delta(-\delta)=0,\ \frac{d}{du}f_+ ^\delta (-\delta)=\mu , \nonumber \\
f_- ^\delta(u) \leq \inf _{|v|\leq \delta} f_- ^\delta(u+v),\ f_- ^\delta(\pm 1-\delta) =0,\ f_- ^\delta(\delta)=0,\ \frac{d}{du}f_- ^\delta (\delta)=\mu , \nonumber
\end{align}
and $u_0^\varepsilon$ satisfies (\ref{eq:ini}). Note that we choose the reaction terms $f_\pm ^\delta$ to satisfy (\ref{reaction2}). Thus we can apply the result of Section \ref{sec2} to the solutions $\bar{u}_\pm ^{\varepsilon , \delta}$. We set a stopping time $\tau _3:=\inf \{ t>0 | |u_1(t,x)|>\delta\ for\ some \  x\in \mathbb{R}\}$.
\begin{lemma}
\label{lem82}
On the event $\{ \omega \in \Omega | \tau_3 \geq 1 \}$, we have that
\begin{align}
\bar{u}_- ^{\varepsilon ,\delta}(t,x) -\delta \leq u^\varepsilon (t,x) \leq \bar{u}_+ ^{\varepsilon ,\delta}(t,x) +\delta, \ \ \ t\in [0,1],\ x \in \mathbb{R}, \nonumber
\end{align}
where $u^\varepsilon$ is the solution of (\ref{eq:spde}).
\end{lemma}
\begin{proof}
We only consider the upper bound on $\{ \omega \in \Omega | \tau_3 \geq 1 \}$. We consider the PDE
\begin{align}
\begin{cases}
\dot{u}_2 (t,x) &= \Delta u_2 (t,x)+\displaystyle{\frac{1}{\varepsilon}}f(u_1 +u_2) ,\ \ \ t > 0,\  x\in \mathbb{R}, \nonumber \\
u_2 (0,x) &= u_0 ^\varepsilon (x),\ \ \ x\in \mathbb{R}, \nonumber
\end{cases}
\end{align}
where $u_0 ^\varepsilon$ satisfies (\ref{eq:ini}), and take the function $v(t,x)= \bar{u}_+ ^{\varepsilon ,\delta}(t,x)-u_2 (t,x)$. Here $u_1$ is defined in the proof of Proposition \ref{thm81}. Note that $u^\varepsilon = u_1 + u_2$. The rest of this proof is similar to that of Lemma 2.2 of \cite{f}.
\qed \end{proof}

\begin{proposition}
\label{thm82}
Let $u^\varepsilon$ be the solution of (\ref{eq:spde}) and assume that the initial value $u_0^\varepsilon$ satisfies (\ref{eq:ini}). Then there exist some positive constants $C_1, \ C>0$ and $K>1$ such that
\begin{align}
\lim _{\varepsilon \downarrow 0}P\left( | u^\varepsilon (t,x) -1 | \leq \varepsilon ^{\kappa } \left( C\exp \left( -\frac{\sqrt{\mu} x}{2}\right) +1 \right ) \ for \ all \ t \in [0, C_1 \varepsilon |\log \varepsilon |],\ x \geq K \right) =1, \nonumber \\
\lim _{\varepsilon \downarrow 0}P\left( | u^\varepsilon (t,x) +1 | \leq \varepsilon ^{\kappa }\left( C\exp \left( \frac{\sqrt{\mu} x}{2}\right) +1 \right ) \ for \ all \ t \in [0, C_1 \varepsilon |\log \varepsilon |],\ x \leq -K \right) =1, \nonumber
\end{align}
for all $0< \kappa <\gamma$.
\end{proposition}
\begin{proof}
We only prove the first one. By taking $\delta =\varepsilon^{\kappa }$ and $K$ as in Proposition \ref{pro21}, we obtain
\begin{align}
&P \left ( | u^\varepsilon (t,x) -1 | \leq \varepsilon ^{\kappa } \left ( C\exp \left ( -\frac{\sqrt{\mu} x}{2} \right ) +1\right ) \ for \ all \ t \in [0, C_1 \varepsilon |\log \varepsilon |],\ x \geq K \right ) \nonumber \\
&\ \ \ \ \ \ \ \ \ \ \ \ \ \ \ \ \ \ \ \ \ \ \geq P(\varepsilon ^\gamma Y(\omega )\leq \varepsilon ^{\kappa }) \geq 1- \varepsilon ^{p (\gamma - \kappa )}E[Y^p] \nonumber
\end{align}
for all $p\geq1$ from Lemma \ref{lem82}, Chebyshev inequality and Proposition \ref{pro21}. We apply Proposition \ref{pro21} for the solutions $\bar{u}_\pm ^{\varepsilon , \delta}$.
\qed \end{proof}

\begin{proposition}
\label{pro81}
Let $\underbar{u}$ and $\bar{u}$ be the solutions of (\ref{eq:spde}) which satisfy $\underbar{u}(0,x) \leq \bar{u}(0,x)$ for all $x \in \mathbb{R}$. Then $\underbar{u}(t,x) \leq \bar{u}(t,x)$ holds for all $x \in \mathbb{R}$ and $t \in [0,\infty )$ $P$-a.s.
\end{proposition}
\begin{proof}
We can show this proposition by applying the maximum principle in a similar way to the proof of Lemma \ref{lem82}.
\qed \end{proof}

\subsection{Energy estimates}
Let $u(t,x)$ be a solution of PDE (\ref{eq:pde}) where $f$ satisfies (\ref{reaction}) and $u^\varepsilon _0$ satisfies (\ref{eq:ini}). We set $t = C_1 \varepsilon |\log \varepsilon |$ which is the generation time of $u$ and the constant $C_1\in (0, \frac{1}{\mu})$ is given in Proposition \ref{pro21}. Because $\kappa >1$, we imediately see that
\begin{align}
\label{est810}
dist(u(C_1\varepsilon |\log \varepsilon |,\cdot ),M^\varepsilon )\leq C\varepsilon ^{\bar{\beta}},
\end{align}
from Proposition \ref{pro21}. Proposition \ref{thm81} and (\ref{est810}) imply that the solution $u^\varepsilon$ of SPDE (\ref{eq:spde}) is in the $C(\varepsilon ^{\bar{\beta}} + \varepsilon ^{\kappa })$-neighborhood of $M^\varepsilon$, though this is not enough. In order to show the main result, we need much better estimates.

We now construct super and sub solutions of $u^\varepsilon$. We see that 
\begin{align}
u^\varepsilon(t,x) \leq \bar{u}_+ ^{\varepsilon ,\delta}(t,x)+\delta \ \ \ t\in [0,T \wedge \tau _3] ,\ x\in \mathbb{R}, \nonumber
\end{align} 
from Lemmas \ref{lem81} and \ref{lem82}. Recall that $\tau _3 := \inf \{t>0| |u_2 (t,x)|  > \delta \ for\ some\ x \in \mathbb{R} \}$. If $\tau _3 \geq C_1 \varepsilon |\log \varepsilon |$, then
\begin{align} 
\bar{u}_+ ^{\varepsilon ,\delta}(C_1 \varepsilon |\log \varepsilon | \wedge \tau _3,x) & = \bar{u}_+ ^{\varepsilon ,\delta}(C_1 \varepsilon |\log \varepsilon | ,x)\nonumber \\
&\leq m(\varepsilon ^{-\frac{1}{2}} (x-\xi_0 - C \varepsilon^{\bar{\beta}})) + C' \varepsilon ^{\kappa } \nonumber
\end{align} 
by applying Proposition \ref{pro21} to $\bar{u}_+ ^{\varepsilon ,\delta}$ for $\delta = \varepsilon ^{\kappa }$ and $\kappa >0$. The function $m$ is defined in (\ref{ode11}). And we see that
\begin{align}
\int _{\mathbb{R} \backslash [-2,2]} | u^\varepsilon (C_1 \varepsilon |\log \varepsilon |,x) - m(\varepsilon ^{-\frac{1}{2}} x)|^2 dx \leq \varepsilon ^{2\kappa } \nonumber
\end{align}
from Propositions \ref{pro21} and \ref{thm81}.

We consider $t= C_1\varepsilon |\log \varepsilon |$ as an initial time. Namely, we consider the SPDE (\ref{eq:spde}) which is replaced $u_0 ^\varepsilon$ by $u ^\varepsilon (C_1\varepsilon |\log \varepsilon |, x)$. We can construct super and sub solutions $u_\pm ^\varepsilon$ for SPDE (\ref{eq:spde}) which satisfy
\begin{align}
\label{in}
\begin{cases}
\text{(i)} u_\pm ^\varepsilon \text{ obey the SPDE (\ref{eq:spde}) for all }t > 0, \\
\text{(ii)} dist(u_\pm ^\varepsilon (t,\cdot),M^\varepsilon )\leq C \varepsilon ^{\kappa }\ (\text{for all }t\in [ 0 , (\frac{1}{\mu}-C_1)\varepsilon |\log \varepsilon |\wedge \tau _2 \wedge \tau _3]).
\end{cases}
\end{align}
Indeed,  by combining the estimates as above, we take an initial value of super solution as follows:
\begin{align}
\label{ini3}
u_+ ^\varepsilon (0,x):=(1 - \chi_1 (x))u^\varepsilon (0,x) + \chi_2 (x) (m(\varepsilon ^{-\frac{1}{2}} (x-\xi_0 - C \varepsilon^{\bar{\beta}})) + C' \varepsilon ^{\kappa }),
\end{align}
where $\chi _1$ and $\chi _2$ are some positive cutoff functions in $C_0 ^\infty (\mathbb{R})$ which take values in $[0,1]$. The function $\chi _1$ takes 1 when $x\in [-1,1]$ and takes 0 when $x\in \mathbb{R} \backslash [-2,2]$. The function $\chi _2$ takes 1 when $x\in [-2,2]$ and takes 0 when $x\in \mathbb{R} \backslash [-3,3]$. We can check easily that $u_+ ^\varepsilon (0,x)\geq u ^\varepsilon (0,x)$ and that $u_+ ^\varepsilon (t,x)$ dominates the solution $u ^\varepsilon (t,x)$ for all $t \in [0,\varepsilon ^{-2\gamma -\frac{1}{2}}T]$ and $x\in \mathbb{R}$ from Proposition \ref{pro81}. Here super solution $u_+ ^\varepsilon (t,x)$ satisfies (ii) of (\ref{in}) because of Lemma 9.1 of \cite{f} and Proposition \ref{thm81} in this section. We can construct $u_- ^\varepsilon$ in a similar way. Indeed, we can take an initial value of $u_- ^\varepsilon$;
\begin{align}
u_- ^\varepsilon (0,x):=(1 - \chi_1 (x))u^\varepsilon (0,x) + \chi_2 (x) (m(\varepsilon ^{-\frac{1}{2}} (x-\xi_0 + C \varepsilon^{\bar{\beta}})) - C' \varepsilon ^{\kappa }).
\end{align}
The functions $\chi_1$ and $\chi_2$ are same as above, and $u_- ^\varepsilon$ also satisfies (ii) of (\ref{in}).

Now we show that $u_\pm ^\varepsilon$ stay in the $\varepsilon ^{\kappa '}$-neighborhood of $M^\varepsilon$ in $L^2$-sense, not only for $t \in [0,(\frac{1}{\mu}-C_1)\varepsilon |\log \varepsilon |]$ but for $t \in [0,\varepsilon ^{-2\gamma -\frac{1}{2}}T]$ for some $1<\kappa ' <\kappa $ with high probability. In order to show this, we prove that $u_\pm ^\varepsilon$ enters the $\varepsilon ^{\kappa '}$-neighborhood of $M^\varepsilon$ in the $H^1$-sense.

We only consider the super solution $u_+ ^\varepsilon$. We change the scale of the solution $u_+ ^\varepsilon$ in time and space variables as below;
\begin{align}
v(t,x) := u_+ ^\varepsilon (\varepsilon ^{-2\gamma-\frac{1}{2}}t,\varepsilon ^{\frac{1}{2}}x),\ \ \ t\in[0,\infty ), \ x\in \mathbb{R}. \nonumber
\end{align}
We define an approximation of the function
$u^\delta (t,x):= (\rho ^{\delta(\cdot)} (\cdot) \ast u(t, \cdot ))(x)$ where $\rho$ is the function satisfying
\begin{align}
 \begin{cases}
  \text{(i)}\int_\mathbb{R} \rho (t)dt=1, & \nonumber \\
  \text{(ii)}\rm{supp} \rho \subset [0,1], & \nonumber \\
  \text{(iii)}\rho \in C^\infty (\mathbb{R} ). \nonumber \\
 \end{cases}
\end{align}
And $\rho^{\delta(x)}$ satisfies the following conditions;
\begin{align}
\begin{cases}
\rho^{\delta(x)} = \frac{1}{\delta}\rho( \frac{x}{\delta}) \ \ \ (|x|\leq \varepsilon ^{-\frac{1}{2}}+1), \nonumber \\
\rho^{\delta(x)} = \frac{1}{\delta(x)}\rho( \frac{x}{\delta(x)}) \ \ \ ( \varepsilon ^{-\frac{1}{2}}+1\leq |x| \leq \varepsilon ^{-\frac{1}{2}} +2), \nonumber \\
\rho^{\delta(x)} = \rho ^0 \ \ \ (|x|\geq \varepsilon ^{-\frac{1}{2}}+2 ), \nonumber
\end{cases}
\end{align}
where we denote $\rho ^0 \ast u =u$ formally, $\delta (\cdot) \in C^\infty (\mathbb{R})$ and  $0 \leq \delta (x) \leq \delta$. We can see the precise conditions of this convolution in Section 4 and 6 of \cite{f}.

Before the proof, we state the SPDE which $v(t,x)$ satisfies in law sense;
\begin{align}
\dot{v}(t,x) = \varepsilon ^{-2\gamma-\frac{3}{2}} \{ \Delta v + f(v) \} + \varepsilon ^{-\frac{1}{2}} a(\varepsilon ^{\frac{1}{2}} x) \dot{W} _t (x). \nonumber
\end{align}
From Proposition \ref{thm81} and the result of the generation of interface for PDE, we only need to show the case that $dist(v(0,\cdot),M) \leq \varepsilon ^{\kappa -\frac{1}{4}}$ from the strong Markov property where $M:=M^1=\{ m(x-\eta ) | \eta \in \mathbb{R} \}$. Moreover we see that there exists unique Fermi coordinate $v_t = s(v_t) + m_{\eta(v_t)}$ if $t \in [0,(\frac{1}{\mu}-C_1)\varepsilon^{2\gamma +\frac{3}{2}} |\log \varepsilon |\wedge \varepsilon^{2\gamma +\frac{1}{2}} \tau_2 \wedge\varepsilon^{2\gamma +\frac{1}{2}} \tau_3]$ from (ii) of (\ref{in}). See Section \ref{sec1} for Fermi coordinate.

\begin{lemma}
\label{lem83}
For each $T>0$, $t\in [0,T\wedge \tau _1]$ and $p>1$ there exist positive random variables $Y^\varepsilon(\omega )$, $Z ^\varepsilon(\omega ) \in L^p(\Omega )$ such that $\sup _{0<\varepsilon <1}E[(Y^\varepsilon )^p] < \infty ,\ \sup _{0<\varepsilon <1}E[(Z ^\varepsilon )^p] < \infty$,
\begin{align}
\| v_t - v_t ^\delta \|_{L^2(\mathbb{R})} \leq Y^\varepsilon \delta + Z^\varepsilon \varepsilon ^{\frac{1}{16} + \frac{3\gamma}{4} -\alpha} \delta ^{\frac{1}{4} -\alpha '} \nonumber
\end{align}
for sufficiently small $\alpha$ and $\alpha ' >0$
\end{lemma}
\begin{proof}
See Lemmas 4.1, 4.2 and the proof of Theorem 5.1 in \cite{f}. We note that the initial value satisfies assumptions for these lemmas and theorem.
\qed \end{proof}

\begin{lemma}
\label{lem84}
We define a stopping time $\tau ^\delta := \inf \{t>0 | \| s(v_t ^\delta) \|_{H^1} \leq \varepsilon ^{\kappa '} \}$ for $\kappa ' >1$. If we can take $\kappa > \kappa '$ which satisfy $(\kappa ' +\frac{21}{40} + \frac{\gamma}{10})\vee 2\kappa ' < \kappa < \gamma -\frac{C_f}{\mu}$ and $1< \kappa ' <\frac{1}{20} + \frac{\gamma}{5}$, then there exists a sequence $\{ \delta _\varepsilon \}$, which converges to $0$ as $\varepsilon \to 0$, such that
\begin{align}
\lim _{\varepsilon \downarrow 0}P\left( \tau ^{\delta_\varepsilon } \leq \varepsilon ^{2\gamma +\frac{3}{2}+\alpha} |\log \varepsilon| \right) =1, \nonumber
\end{align}
for sufficiently small $\alpha >0$.
\end{lemma}
\begin{proof}
At first, we fix a time $t \in [0,(\frac{1}{\mu}-C_1) \varepsilon ^{2\gamma +\frac{3}{2}} |\log \varepsilon| \wedge \varepsilon^{2\gamma +\frac{1}{2}} \tau_2 \wedge \varepsilon^{2\gamma +\frac{1}{2}} \tau_3]$ at which $u^\varepsilon (\varepsilon^{-2\gamma-\frac{1}{2}}t)$ satisfies the conditions of Proposition \ref{thm82}. From the definition of Fermi coordinate, we get
\begin{align}
\label{est8-2}
\| s(v_t ^\delta) \|_{H^1} &= \| v_t ^\delta -m_{\eta(v_t ^\delta)} \|_{H^1} \nonumber \\
&\leq \| v_t ^\delta -m_{\eta(v_t )} ^\delta \|_{H^1} + \| m_{\eta(v_t )} ^\delta - m_{\eta(v_t )} \|_{H^1} + \| m_{\eta(v_t )} -m_{\eta(v_t ^\delta )} \|_{H^1}\nonumber \\
&\leq \|s^\delta (v_t)\|_{H^1} + C\delta + C' \| v_t - v_t ^\delta \|_{L^2},
\end{align}
for $t\leq (\frac{1}{\mu}-C_1) \varepsilon ^{2\gamma +\frac{3}{2}} |\log \varepsilon| \wedge \varepsilon^{2\gamma +\frac{1}{2}} \tau_2 \wedge \varepsilon^{2\gamma +\frac{1}{2}} \tau_3$. We just use the triangle inequality for the second line. We denote the approximation of $s(v_t)$ convoluted with $\rho ^\delta$ as $s^\delta (v_t)$ in the third line. We can estimate the second term of the second line by the integrability and the differentiability of $m_{\eta(v_t )} ^\delta - m_{\eta(v_t )}$ (see Lemma 5.4 of \cite{f}). From Lemma 5.5 of \cite{f} and the straight calculation, we get the estimate of the third term in the third line. And thus, we need to derive the estimate of $\|s^\delta (v_t)\|_{H^1}$ because Lemma \ref{lem83} completes these estimates if we take $\delta _\varepsilon =\varepsilon ^{\frac{1}{10}+\frac{2\gamma}{5}}$ which is the same $\delta$ as the case in Section 5 of \cite{f}. Now we consider the estimate in $L^2$-norm. An easy computation gives us
\begin{align}
\|s^\delta (v_t)\|_{L^2} &\leq \|s^\delta (v_t)-s(v_t) \|_{L^2} + \|s(v_t)\|_{L^2}\nonumber \\
&\leq \| v_t ^\delta - v_t \|_{L^2} + \| m_{\eta(v_t )} ^\delta - m_{\eta(v_t )} \|_{L^2} + \|s(v_t)\|_{L^2}\nonumber \\
&\leq \| v_t ^\delta - v_t \|_{L^2} + C\delta + \varepsilon ^{\kappa -\frac{1}{4}}. \nonumber
\end{align}
We use the triangle inequality and the definition of Fermi coordinate throughout these estimate. And thus Lemma \ref{lem83} and order of $\delta _\varepsilon$ complete the estimate. Next we need to consider $\|\nabla s^\delta (v_t)\|_{L^2}$ where $\nabla$ means $\frac{d}{dx}$. We divide the integration $\|\nabla s^\delta (v_t)\|_{L^2} ^2$ into four parts as below.
\begin{align}
\label{est8-1}
\int_\mathbb{R} (\nabla s^\delta (v_t) )^2 dx = &\int_{I_\varepsilon ^- \cup I_\varepsilon ^+}(\nabla s (v_t) )^2 dx + \int_{-\varepsilon ^{-\frac{1}{2}}-1} ^{\varepsilon ^{-\frac{1}{2}}+1} (\nabla s^\delta (v_t) )^2 dx\nonumber \\
&+\left( \int_{-\varepsilon ^{-\frac{1}{2}}-2} ^{-\varepsilon ^{-\frac{1}{2}}-1} + \int_{\varepsilon ^{-\frac{1}{2}}+1} ^{\varepsilon ^{-\frac{1}{2}}+2}\right) \{ (\nabla s^\delta (v_t) )^2 - (\nabla s(v_t) )^2 \}dx,
\end{align}
where $I_\varepsilon ^- := (-\infty,-\varepsilon ^{-\frac{1}{2}}-1]$ and $I_\varepsilon ^+ := [\varepsilon ^{-\frac{1}{2}}+1,\infty)$. At first, we derive the estimate for the second term of right hand side of (\ref{est8-1}). From the definition of $s^\delta (v_t)$, $\delta _\varepsilon$ and $\kappa $, a straight calculation gives us
\begin{align}
\int_{-\varepsilon ^{-\frac{1}{2}}-1} ^{\varepsilon ^{-\frac{1}{2}}+1} (\nabla s^\delta (v_t) )^2 dx &= \int_{-\varepsilon ^{-\frac{1}{2}}-1} ^{\varepsilon ^{-\frac{1}{2}}+1} ((\nabla \rho ^\delta) \ast s(v_t) )^2 dx \nonumber \\
&\leq C(\varepsilon ^{-\frac{1}{2}} +1)\delta^{-\frac{1}{2}}\int_\mathbb{R} ( s(v_t) )^2 dx \leq C(\varepsilon ^{-\frac{1}{2}} +1)\delta^{-\frac{1}{2}} \varepsilon ^{2\kappa -\frac{1}{2}} \leq\varepsilon ^{2\kappa '} . \nonumber
\end{align}
Next we consider the last term of (\ref{est8-1}). Note that $s^\delta (v_t)$ and $s (v_t)$ are both differentiable if $|x|\geq \varepsilon ^{-\frac{1}{2}}+1$. Reminding the estimate of $ \| m_{\eta(v_t )} ^\delta - m_{\eta(v_t )} \|_{H^1}$ in (\ref{est8-2}), we can assert that the last term is of order $O(\delta ^2)$. At last, we show that the value of the first term of (\ref{est8-1}) becomes less than $\varepsilon ^{2\kappa '}$ before the time of order $O(\varepsilon ^{2\gamma +\frac{3}{2}} |\log \varepsilon |)$ on the event $\{ (\frac{1}{\mu}-C_1) \varepsilon ^{2\gamma +\frac{3}{2}} |\log \varepsilon| < \varepsilon^{2\gamma +\frac{1}{2}}\tau_2 \wedge \varepsilon^{2\gamma +\frac{1}{2}} \tau_3 \}$. This completes the estimate. Here we only consider the problem on the domain $I_\varepsilon ^+$. We regard $v_t$ as the solution of the boundary value problem
\begin{align}
\begin{cases}
\dot{v}(t,x) = \varepsilon ^{-2\gamma-\frac{3}{2}} \{ \Delta v + f(v) \} , \nonumber \\
v(t,\varepsilon ^{-\frac{1}{2}}+1)=V_t(\omega ),\ \ \ v(0,x)=v_0(x), \nonumber
\end{cases}
\end{align}
for each fixed $\omega \in \Omega$. We note that the boundary value $V_t(\omega )$ is almost 1 for all $t \in [0,\tau _3]$ from the observation in Proposition \ref{thm82}. From Proposition \ref{thm82} and the condition of the solution $m$, $\| m-m_{\eta (v_t)}\|_{L^2(I_\varepsilon ^+)}$ decays as $\varepsilon \to 0$, and its order is $O(\exp (-\frac{C}{\varepsilon}))$ (see Lemma 2.1 of \cite{ham}). And thus, this integral is negligible. The triangle inequality allows us to consider only $s_t:= v_t - m$. For simplicity we use the notation $\nabla$, $\Delta$ and $\partial _t$. From the form of the PDE, for all $T\in [0,(\frac{1}{\mu}-C_1) \varepsilon ^{2\gamma +\frac{3}{2}} |\log \varepsilon| \wedge \varepsilon^{2\gamma +\frac{1}{2}} \tau_2 \wedge \varepsilon^{2\gamma +\frac{1}{2}} \tau _3]$, we obtain
\begin{align}
\int _0^T \int_{I_\varepsilon ^+}\partial _t s(t,x) s(t,x) dx dt =  \varepsilon ^{-2\gamma-\frac{3}{2}} &\int _0^T \int_{I_\varepsilon ^+}\Delta s(t,x) s(t,x) dx dt \nonumber \\
&+  \varepsilon ^{-2\gamma-\frac{3}{2}} \int _0^T \int_{I_\varepsilon ^+}\{ f(v(t,x))-f(m(x)) \} s(t,x) dx dt, \nonumber
\end{align}
and the integration by parts, the equality $\partial _t s(t,x) s(t,x)=\frac{1}{2}\partial _t (s(t,x))^2$ and the boundary condition $v(t,\infty)=1$ give us
\begin{align}
\int _0^T \int_{I_\varepsilon ^+} ( \nabla s(t,x) )^2 dx dt\leq \frac{\varepsilon^{2\gamma+\frac{3}{2}}}{2} \|s_0 \|_{L^2(I_\varepsilon ^+)} ^2 &+ \int _0^T \int_{I_\varepsilon ^+}\{ f(v(t,x))-f(m(x)) \} s(t,x) dx dt\nonumber \\
&-\int _0^T \{ \nabla s(t,\varepsilon ^{-\frac{1}{2}}+1) \} s(t,\varepsilon ^{-\frac{1}{2}}+1) dt. \nonumber
\end{align}
We can derive the same estimate for $I_\varepsilon ^-$ in a similar way, and hence we obtain
\begin{align}
\int _0^T \int_{I_\varepsilon ^+ \cup I_\varepsilon ^-} ( \nabla s(t,x) )^2 dx dt\leq \frac{\varepsilon^{2\gamma+\frac{3}{2}}}{2} \|s_0 \|_{L^2(I_\varepsilon ^+ \cup I_\varepsilon ^-)} ^2 &+ \int _0^T \int_{I_\varepsilon ^+ \cup I_\varepsilon ^-}\{ f(v(t,x))-f(m(x)) \} s(t,x) dx dt\nonumber \\
&-\int _0^T \{ \nabla s(t,\varepsilon ^{-\frac{1}{2}}+1) \} s(t,\varepsilon ^{-\frac{1}{2}}+1) dt\nonumber \\
&+\int _0^T \{ \nabla s(t,-\varepsilon ^{-\frac{1}{2}}-1) \} s(t,-\varepsilon ^{-\frac{1}{2}}-1) dt.\nonumber
\end{align}
The first term of the right hand side is obviously dominated by $\varepsilon ^{2\gamma+\frac{1}{2}+2\kappa }$ because we consider the initial value as (\ref{ini3}). The second term is negative because $v(t,x)$ and $m(x)$ are almost 1 if $|x|\geq \varepsilon ^{-\frac{1}{2}}+1$ from Proposition \ref{thm82}. From Lemma 6.2 in \cite{f}, $\nabla s(t,\varepsilon ^{-\frac{1}{2}}+1)$ and $ \nabla s(t,-\varepsilon ^{-\frac{1}{2}}-1)$ are bounded for all $t \in [0,\tau _2]$. And thus, the third and fourth terms are dominated by $CT \varepsilon ^{\kappa}$ ($0< \kappa <\gamma$) from Proposition \ref{thm82}. To sum up all of these estimation, we get
\begin{align}
T \underset{t\in[0,T]}{\inf} \int_{I_\varepsilon ^+ \cup I_\varepsilon ^-} ( \nabla s(t,x) )^2 dx &\leq \int _0^{T} \int_{I_\varepsilon ^+ \cup I_\varepsilon ^-} ( \nabla s(t,x) )^2 dx dt \nonumber \\
&\leq C \varepsilon ^{2\gamma+\frac{1}{2}+2\kappa } + CT \varepsilon ^{\kappa} , \nonumber
\end{align}
on the event $\{ T < \varepsilon^{2\gamma +\frac{1}{2}}\tau_2 \wedge \varepsilon^{2\gamma +\frac{1}{2}}\tau_3 \}$. We now take $T:=  \varepsilon ^{2\gamma +\frac{3}{2}+\alpha}|\log \varepsilon |$. From this estimate, we see that the integral $\int_{I_\varepsilon ^+ \cup I_\varepsilon ^-} ( \nabla s(t,x) )^2 dx$ becomes less than $\varepsilon ^{2 \kappa '}$ at some time $T_\varepsilon (\omega) \leq O(\varepsilon ^{2\gamma +\frac{3}{2} +\alpha } |\log \varepsilon |)$ $P$-a.s. for sufficiently small $\alpha >0$. Note that $P(\varepsilon^{2\gamma +\frac{1}{2}} \tau_2 \wedge \varepsilon^{2\gamma +\frac{1}{2}} \tau_3 \geq (\frac{1}{\mu}-C_1) \varepsilon ^{2\gamma +\frac{3}{2}} |\log \varepsilon |) \to 1$ as $\varepsilon \to 0$. This is the conclusion of the lemma.
\qed \end{proof}

\begin{lemma}
\label{lem85}
If we can take $\kappa > \kappa ' >1$ which satisfy $(\kappa' +\frac{21}{40} + \frac{\gamma}{10})\vee 2\kappa' < \gamma -\frac{C_f}{\mu}$ and $1< \kappa ' <\frac{1}{20} + \frac{\gamma}{5}$, then there exist a positive random variable $\widetilde{C}(\omega ) \in L^\infty (\Omega )$ and a sequence $\{ \delta _\varepsilon \}$ which converges to $0$ as $\varepsilon \to 0$ such that
\begin{align}
\lim _{\varepsilon \downarrow 0}P\left( \| s(v_t ^{\delta _\varepsilon}) \|_{H^1} \leq \varepsilon ^{\kappa '} \ for \ all \ t \in [\widetilde{C}(\omega ) \varepsilon^{2\gamma +\frac{3}{2}} |\log \varepsilon |,T] \right) =1. \nonumber
\end{align}
\end{lemma}
\begin{proof}
The same argument as in the proof of Proposition 5.4 in \cite{f} (from p.241 to 244) completes the proof of this lemma from Lemma \ref{lem84}. The positive random variable $\widetilde{C}(\omega)$ can be taken as $\widetilde{C}(\omega)=(\varepsilon^{-2\gamma -\frac{3}{2}} |\log \varepsilon |^{-1} \tau_{\delta _\varepsilon}) \wedge 1$ and it is in the class of $L^\infty (\Omega)$ from Lemma \ref{lem84}.
\qed \end{proof}

\begin{corollary}
\label{cor85}
If we can take $\kappa > \kappa ' >1$ which satisfy $(\kappa' +\frac{21}{40} + \frac{\gamma}{10})\vee 2\kappa' < \gamma -\frac{C_f}{\mu}$ and $1< \kappa ' <\frac{1}{20} + \frac{\gamma}{5}$, then there exists a positive random variable $\widetilde{C}(\omega ) \in L^\infty (\Omega )$ such that
\begin{align}
\lim _{\varepsilon \downarrow 0}P\left( s(v_t) \in H^\alpha \ for \ all \ t \in [\widetilde{C}(\omega ) \varepsilon^{2\gamma +\frac{3}{2}} |\log \varepsilon |,T] \right) =1, \nonumber
\end{align}
for all $0< \alpha < \frac{1}{4}$.
\end{corollary}
\begin{proof}
We get this corollary from Lemma \ref{lem85} and the same argument as Lemma 5.6 of \cite{f}.
\qed \end{proof}

\begin{proposition}
\label{thm83}
If we can take $\kappa > \kappa ' >1$ which satisfy $(\kappa' +\frac{21}{40} + \frac{\gamma}{10})\vee 2\kappa' < \gamma -\frac{C_f}{\mu}$ and $1< \kappa ' <\frac{1}{20} + \frac{\gamma}{5}$, then there exists a positive random variable $\widetilde{C}(\omega ) \in L^\infty (\Omega )$ such that
\begin{align}
\lim _{\varepsilon \downarrow 0}P \left( dist (v_t, M)\leq \varepsilon ^{\kappa '} \  for\  all\  t \in [\widetilde{C}(\omega)\varepsilon^{2\gamma +\frac{3}{2}} |\log \varepsilon |,T] \right) =1. \nonumber
\end{align}
\end{proposition}
\begin{proof}
We prove this from Lemma \ref{lem85} in the same way as the proof of Theorem 5.1 in \cite{f}.
\qed \end{proof}

We get similar results for sub solution $u_- ^\varepsilon$ in the similar way to the proofs of Corollary \ref{cor85} and Proposition \ref{thm83}. We set $w(t,x) := u_- ^\varepsilon (\varepsilon ^{-2\gamma-\frac{1}{2}}t,\varepsilon ^{\frac{1}{2}}x)$ for all $t\in[0,\infty )$ and $x\in \mathbb{R}$.

\begin{corollary}
\label{cor86}
If we can take $\kappa > \kappa ' >1$ which satisfy $(\kappa' +\frac{21}{40} + \frac{\gamma}{10})\vee 2\kappa' < \gamma -\frac{C_f}{\mu}$ and $1< \kappa ' <\frac{1}{20} + \frac{\gamma}{5}$, then there exists a positive random variable $\widetilde{C}(\omega ) \in L^\infty (\Omega )$ such that
\begin{align}
\lim _{\varepsilon \downarrow 0}P\left( s(w_t) \in H^\alpha \ for \ all \ t \in [\widetilde{C}(\omega ) \varepsilon^{2\gamma +\frac{3}{2}} |\log \varepsilon |,T] \right) =1, \nonumber
\end{align}
for all $0< \alpha < \frac{1}{4}$.
\end{corollary}

\begin{proposition}
\label{thm84}
If we can take $\kappa > \kappa ' >1$ which satisfy $(\kappa' +\frac{21}{40} + \frac{\gamma}{10})\vee 2\kappa' < \gamma -\frac{C_f}{\mu}$ and $1< \kappa ' <\frac{1}{20} + \frac{\gamma}{5}$, then there exists a positive random variable $\widetilde{C}(\omega ) \in L^\infty (\Omega )$ such that
\begin{align}
\lim _{\varepsilon \downarrow 0}P \left( dist (w_t, M)\leq \varepsilon ^{\kappa '} \  for\  all\  t \in [\widetilde{C}(\omega )\varepsilon^{2\gamma +\frac{3}{2}} |\log \varepsilon |,T] \right) =1. \nonumber
\end{align}
\end{proposition}

From these results, we get the dynamics of $u_\pm ^\varepsilon$ from the same argument as in the case $v_0 \in M$ (see Sections 7 and 8 of \cite{f}). 

\begin{proof}[Proof of Theorem \ref{thm23}]
At first, we note that the condition (\ref{gamma}) assures that all of lemmas, propositions and corollaries in this section hold.

Now we construct super and sub solutions again. We consider $C_1\varepsilon |\log \varepsilon |$ as the initial time $0$. From (\ref{in}), we can see $u_- ^\varepsilon (0,x)\leq u_0 ^\varepsilon (x) \leq u_+ ^\varepsilon(0,x)$, and Proposition \ref{pro81} allows us to compare the solutions as below:
\begin{align}
\label{est:compa}
u_- ^\varepsilon (t,x)\leq u^\varepsilon (t,x) \leq u_+ ^\varepsilon (t,x)\ for \ all \ t\in[0,\varepsilon ^{-2\gamma -\frac{1}{2}}T] \ and \ x \in \mathbb{R}.
\end{align}
Because $dist(u_\pm ^\varepsilon(0,\cdot) ,M^\varepsilon )\leq C\varepsilon ^{\kappa }$, by taking $\bar{u}_\pm ^\varepsilon(t,x) := u_\pm ^\varepsilon(\varepsilon ^{-2\gamma -1}t,x)$ and $\kappa  >1$, we can show that 
\begin{align}
P\left ( \underset{ t \in [0,T] }{ \sup } \| \bar{u}_\pm ^\varepsilon(t,\cdot )-\chi _{\xi^{\varepsilon ,\pm} _t}(\cdot )\| _{L^2(\mathbb{R})} \leq \delta \right ) \to 1\ \ \ (\varepsilon \to 0) \nonumber
\end{align}
in the same way as \cite{f}. The stochastic processes $\xi^{\varepsilon ,\pm} _t$ can be defined in the same way as (8.3) of \cite{f} (replace $v^\varepsilon$ in (8.3) by $v$ and $w$). We see that $|\xi^{\varepsilon ,\pm} _0-\xi_0| \leq \varepsilon^{\bar{\beta}} $ and the difference of initial value does not matter for the proof of tightness for $\{ P_\pm ^\varepsilon \}$ which are the distributions of $\{ \xi _t ^{\varepsilon ,\pm} \}$ on $C([0,T],\mathbb{R})$. Here $\bar{\beta}>0$ is a constant which is defined in Proposition \ref{pro21}. Thus the distribution of the process $\{ \xi^{\varepsilon ,\pm} _t\}$  on $C([0,T],\mathbb{R})$ converges to that of $\{ \xi _t\}$ weakly as $\varepsilon \to 0$. Therefore we obtain
\begin{align}
&\| \bar{u}_\varepsilon (t,\cdot )-\chi _{\xi^{\varepsilon } _t}(\cdot )\| _{L^2(\mathbb{R})} \nonumber \\
&\leq \| \bar{u}^\varepsilon (t,\cdot )-\bar{u}_+ ^\varepsilon (t,\cdot )\| _{L^2} + \| \bar{u}_+ ^\varepsilon (t,\cdot )-\chi _{\xi^{\varepsilon ,+ } _t}(\cdot )\| _{L^2} \nonumber \\
&\leq \| \bar{u}_+ ^\varepsilon (t,\cdot )-\bar{u}_- ^\varepsilon (t,\cdot )\| _{L^2} + \| \bar{u}_+ ^\varepsilon (t,\cdot )-\chi _{\xi^{\varepsilon ,+ } _t}(\cdot )\| _{L^2} \nonumber \\
&\leq 2 \| \bar{u}_+ ^\varepsilon (t,\cdot )-\chi _{\xi^{\varepsilon ,+ } _t}(\cdot )\| _{L^2} + \| \chi _{\xi^{\varepsilon ,+ } _t}(\cdot ) - \chi _{\xi^{\varepsilon ,- } _t}(\cdot )\| _{L^2} + \| \chi _{\xi^{\varepsilon ,- } _t}(\cdot )-\bar{u}_- ^\varepsilon (t,\cdot )\| _{L^2}, \nonumber
\end{align}
for all $t \in [0,T\wedge \varepsilon ^{2\gamma +\frac{1}{2}} (\tilde{\sigma} _\varepsilon \wedge \bar{\sigma} _\varepsilon)]$, by taking $\xi ^\varepsilon _t := \xi^{\varepsilon ,+} _t$. Here we set the stopping time
\begin{align}
&\tilde{\sigma} _\varepsilon := \inf \{t>0 | dist (v_t ,M)> \varepsilon ^{\kappa '} \ or \ \|v_t\|_{L^\infty}>2C_0 \ or \ v_t \notin H^\alpha +m \} \nonumber \\
&\bar{\sigma} _\varepsilon := \inf \{t>0 | dist (w_t ,M)> \varepsilon ^{\kappa '} \ or \ \|w_t\|_{L^\infty}>2C_0 \ or \ w_t \notin H^\alpha +m \} \nonumber
\end{align}
for fixed $\alpha <\frac{1}{4}$. Indeed, the first and the third inequality come from the triangle inequality, and (\ref{est:compa}) gives us the second inequality. From Propositions \ref{thm31}, \ref{thm83}, \ref{thm84}, Corollaries \ref{cor85} and \ref{cor86}, we see that $P(T \leq \varepsilon ^{2\gamma +\frac{1}{2}} (\tilde{\sigma} _\varepsilon \wedge \bar{\sigma} _\varepsilon)) \to 1$ as $\varepsilon \to 0$. This completes the proof of the theorem by taking $C(\omega):=C_1 + \widetilde{C}(\omega)$. Here, $C_1$ is introduced in Proposition \ref{pro21}.
\qed \end{proof}

\ \\
{\bf Acknowledgements}

The author would like to thank Professor T. Funaki for his tremendous supports and incisive advices. This work was supported by the Program for Leading Graduate Schools, MEXT, Japan and Japan society for the promotion of science, JSPS.



\end{document}